\documentclass[11pt]{article}
\RequirePackage{amsthm,amsmath,amsfonts,amssymb}
\usepackage[authoryear]{natbib}
\usepackage{amsfonts}
\usepackage{bbm}
\usepackage{mathrsfs}
\usepackage{subfigure}
\usepackage[dvipdfmx]{graphicx} 
\usepackage{titlesec}
\usepackage{lipsum}
\usepackage{enumerate}
\usepackage{color}
\usepackage{amssymb}
\usepackage{amsmath,amssymb,amsfonts,amsthm}
\usepackage{bm}
\usepackage{here}
\usepackage{multicol}
\usepackage{epsf}
\usepackage{authblk}
\makeatother
\begin{document}
\newcommand{\field}[1]{\mathbb{#1}} 
\newcommand{\ind}{\mathbbm{1}}
\newcommand{\C}{\mathbb{C}}
\newcommand{\D}{\,\mathscr{D}}
\newcommand{\E}{\,\mathrm{E}}
\newcommand{\Prob}{\,\mathrm{P}}
\newcommand{\F}{\,\mathscr{F}}
\newcommand{\I}{\,\mathrm{i}}
\newcommand{\N}{\field{N}}
\newcommand{\T}{\,\mathrm{T}}
\newcommand{\Ls}{\,\mathscr{L}}
\newcommand{\ve}{\,\mathrm{vec}}
\newcommand{\var}{\,\mathrm{var}}
\newcommand{\cov}{\,\mathrm{Cov}}
\newcommand{\vech}{\,\mathrm{vech}}
\newcommand*\dif{\mathop{}\!\mathrm{d}}
\newcommand{\difs}{\mathrm{d}}
\newcommand\mi{\mathrm{i}}
\newcommand\me{\mathrm{e}}
\newcommand{\R}{\field{R}}

\newcommand{\es}{\hat{f}_n}
\newcommand{\ess}{\hat{f}_{\flr{ns}}}
\newcommand{\p}[1]{\frac{\partial}{\partial#1}}
\newcommand{\pp}[1]{\frac{\partial^2}{\partial#1\partial#1^{\top}}}
\newcommand{\para}{\bm{\theta}}
\providecommand{\bv}{\mathbb{V}}
\providecommand{\bu}{\mathbb{U}}
\providecommand{\bt}{\mathbb{T}}
\newcommand{\flr}[1]{\lfloor#1\rfloor}
\newcommand{\ba}{B_m}
\newcommand{\bxi}{\bar{\xi}_n}
\newcommand{\sgn}{{\rm sgn \,}}
\newcommand{\rint}{\int^{\infty}_{-\infty}}

\newcommand{\dr}{\mathrm{d}}
\newcommand{\red}[1]{\textcolor{red}{#1}}

\newcommand{\skakko}[1]{\left(#1\right)}
\newcommand{\mkakko}[1]{\left\{#1\right\}}
\newcommand{\lkakko}[1]{\left[#1\right]}

\newcommand{\Z}{\field{Z}}
\newcommand{\Zo}{\field{Z}_0}

\newcommand{\abs}[1]{\lvert#1\rvert}
\newcommand{\ct}[1]{\langle#1\rangle}
\newcommand{\inp}[2]{\langle#1,#2\rangle}
\newcommand{\norm}[1]{\lVert#1 \rVert}
\newcommand{\Bnorm}[1]{\Bigl\lVert#1\Bigr  \rVert}
\newcommand{\Babs}[1]{\Bigl \lvert#1\Bigr \rvert} 
\newcommand{\ep}{\epsilon} 
\newcommand{\sumn}[1][i]{\sum_{#1 = 1}^T}
\newcommand{\tsum}[2][i]{\sum_{#1 = -#2}^{#2}}
\providecommand{\abs}[1]{\lvert#1\rvert}
\providecommand{\Babs}[1]{\Bigl \lvert#1\Bigr \rvert} 

\newcommand{\uint}{\int^{1}_{0}}
\newcommand{\freqint}{\int^{\pi}_{-\pi}}
\newcommand{\li}[1]{\mathfrak{L}(S_{#1})}

\newcommand{\cum}{{\rm cum}}

\newcommand{\xt}{\bm{X}_{t, T}}
\newcommand{\yt}{\bm{Y}_{t, T}}
\newcommand{\zt}{\bm{Z}_{t, T}}
\newcommand{\gcu}{{\rm GC}^{2 \to 1}(u)}

\newcommand{\btheta}{\bm{\bm{\theta}}}
\newcommand{\bbeta}{\bm{\eta}}
\newcommand{\bzeta}{\bm{\zeta}}
\newcommand{\bzero}{\bm{0}}
\newcommand{\bI}{\bm{I}}
\newcommand{\bd}{\bm{d}}
\newcommand{\bx}[1]{\bm{X}_{#1, T}}
\newcommand{\be}{\bm{\ep}}
\newcommand{\bp}{\bm{\phi}}

\newcommand{\act}{A_T^{\circ}}
\newcommand{\ac}{A^{\circ}}

\newcommand{\dlim}{\xrightarrow{d}}
\newcommand{\plim}{\rightarrow_{P}}

\newcommand{\ls}{\mathcal{S}}
\newcommand{\cs}{\mathcal{C}}

\providecommand{\ttr}[1]{\textcolor{red}{ #1}}
\providecommand{\ttb}[1]{\textcolor{blue}{ #1}}
\providecommand{\ttg}[1]{\textcolor{green}{ #1}}
\providecommand{\tty}[1]{\textcolor{yellow}{ #1}}
\providecommand{\tto}[1]{\textcolor{orange}{ #1}}
\providecommand{\ttp}[1]{\textcolor{purple}{ #1}}

\newcommand{\sign}{\mathop{\rm sign}}
\newcommand{\conv}{\mathop{\rm conv}}
\newcommand{\argmax}{\mathop{\rm arg~max}\limits}
\newcommand{\argmin}{\mathop{\rm arg~min}\limits}
\newcommand{\argsup}{\mathop{\rm arg~sup}\limits}
\newcommand{\arginf}{\mathop{\rm arg~inf}\limits}
\newcommand{\diag}{\mathop{\rm diag}}
\newcommand{\minimize}{\mathop{\rm minimize}\limits}
\newcommand{\maximize}{\mathop{\rm maximize}\limits}
\newcommand{\tr}{\mathop{\rm tr}}
\newcommand{\Cum}{\mathop{\rm Cum}\nolimits}
\newcommand{\Var}{\mathop{\rm Var}\nolimits}
\newcommand{\Cov}{\mathop{\rm Cov}\nolimits}

\numberwithin{equation}{section}
\theoremstyle{plain}
\newtheorem{thm}{Theorem}[section]

\newtheorem{lem}[thm]{Lemma}
\newtheorem{prop}[thm]{Proposition}
\theoremstyle{definition}
\newtheorem{defi}[thm]{Definition}
\newtheorem{assumption}[thm]{Assumption}
\newtheorem{cor}[thm]{Corollary}
\newtheorem{rem}[thm]{Remark}
\newtheorem{eg}{Example}

\title{Two step estimations via the Dantzig selector  for models of stochastic processes with high-dimensional parameters}
\author[1]{Kou Fujimori}
\author[2]{Koji Tsukuda}
\affil[1]{Faculty of Economics and Law,
Shinshu University.}
\affil[2]{Faculty of Mathematics, Kyushu University.}
\date{}
\maketitle
\begin{abstract}
We consider the sparse estimation for stochastic processes with possibly infinite-dimensional nuisance parameters, by using 
the Dantzig selector which is a sparse estimation method similar to $Z$-estimation.
When a consistent estimator for a nuisance parameter is obtained, it is possible to construct an asymptotically normal 
estimator for the parameter of interest under appropriate conditions.
Motivated by this fact, we establish the 
asymptotic behavior of the Dantzig selector 
for models of ergodic stochastic processes 
with high-dimensional parameters of interest and 
possibly infinite-dimensional nuisance parameters.
Moreover, we construct an asymptotically normal estimator by the two step estimation with help of the variable selection through the Dantzig selector 
and a consistent estimator
of the nuisance parameter.
Applications to ergodic time
series models including integer-valued autoregressive models and ergodic diffusion processes are
presented.
\end{abstract}
\section{Introduction}\label{sec:introduction}
High-dimensional modeling in statistics has been attracted much attention over two decades.
Specifically, the sparse estimation methods for 
high-dimensional parameter have been studied and 
commonly used in various fields.
Most of sparse estimation procedures such as  the Lasso introduced by \cite{tibshirani1996} are constructed by adding penalty terms to contrast functions in order to induce the sparsity of the parameters.
On the other hand, the Dantzig selector proposed by 
\cite{candes2007dantzig} is defined as inequality constrained optimization problems.
In this paper, we focus on the Dantzig selector 
type estimator for models of stochastic processes in semiparametric setting.

Let us begin with the following conditional heteroskedastic 
linear regression model:
\[
{Y}_i = \bm{\bm{\theta}}^\top \bm{Z}_i + \epsilon_i,\quad
i=1,\ldots,n,
\]
where $\bm{\bm{\theta}} \in \mathbb{R}^p$ is a regression coefficient, $\bm{Z}_i, i=1,\ldots,n$ are 
$\mathbb{R}^p$-valued covariates and 
$\epsilon_i, i=1,\ldots,n$ are independent random variables with $\E[\epsilon_i|\bm{Z}_i] = 0$, a.s.~and we put
$\Var[\epsilon_i| \bm{Z}_i] = \sigma^2(\bm{Z}_i)$.
The Dantzig selector type estimator $\hat{\bm{\bm{\theta}}}_D$ for 
$\bm{\bm{\theta}}$ is then defined as follows:
\[\hat{\bm{\bm{\theta}}}_D :=\argmin \limits_{\bm{\bm{\theta}} \in \mathcal{C}} \|\bm{\bm{\theta}}\|_1,\quad \mathcal{C} :=\{\bm{\bm{\theta}} \in \mathbb{R}^p : \|\psi_n^{(1)}(\bm{\bm{\theta}})\|_\infty \leq \lambda\},\]
where $\lambda \geq 0$ is a tuning parameter and
\[
\psi_n^{(1)}(\bm{\bm{\theta}}) := \frac{1}{n} \sum_{i=1}^n \bm{Z}_i (Y_i - \bm{\bm{\theta}}^\top \bm{Z}_i)
\]
is a score function based on the least squares method.
When $p$ is large and the true value $\bm{\bm{\theta}}_0$
is sparse, $\hat{\bm{\bm{\theta}}}_D$ behaves better than 
the ordinary least squares estimator 
under some technical conditions, such as 
restricted eigenvalue condition for Hessian matrix and sub-Gaussian property of the noise 
distribution; see, e.g., \cite{bickel2009simultaneous}
for details.
Using $\hat{\bm{\bm{\theta}}}_D$, we can consider the 
estimator $\hat{T}_D$ for true support index set 
$T_0 := \{j : \bm{\theta}_{0j}\not=0\}$, which enables us to 
consider the dimension reduction.
On the other hand, we can also consider the 
following score function:
\[
\Psi_n(\bm{\bm{\theta}}, \sigma^2)
:= \frac{1}{n}\sum_{i=1}^n \frac{\bm{Z}_i ({Y}_i - \bm{Z}_i^\top \bm{\bm{\theta}})}{\sigma^2(\bm{Z}_i)}.
\]
Then, if we can construct a consistent estimator 
$\hat{\sigma}^2(\cdot)$ of $\sigma^2(\cdot)$ under some conditions, the estimator 
$\tilde{\bm{\bm{\theta}}}$ which is a solution to the following 
equation
\[
\Psi_n(\bm{\bm{\theta}}, \hat{\sigma}^2)\approx \bm{0}
\]
is more efficient than the 
ordinary least squares estimator in low-dimensional settings;
see, e.g., \cite{robinson1987asymptotically}.
Therefore, even in high-dimensional, but sparse settings, 
we may construct more efficient estimator by the following steps:
First, construct an initial estimator and consider the variable selection 
via the Dantzig selector based on the score
function $\psi_n^{(1)}$
for a high-dimensional parameter of interest. 
Second, plugging a consistent estimator
for nuisance parameters in the score function 
$\Psi_n$, we construct the $Z$-estimator 
of the unknown parameter, which may be more efficient 
than the ordinary least squares estimator.
Inspired by this ideas, we deal with 
semiparametric models of stochastic processes
and time series 
in high-dimensional settings.

In Section \ref{sec:time series}, we consider the following 
conditionally heteroskedastic model of $1$-dimensional time series with high-dimensional parameters $\bm{\alpha}$ and $\bm{\beta}$:
\[
X_t = S(\bm{\bm{\alpha}}^\top\phi(\bm{X}_{t-1}), \bm{\beta}^\top \bm{Z}_{t-1})
+ u_t,\quad t \in \mathbb{Z},
\]
where 
\[
\E[u_t | \mathcal{F}_{t-1}] =0,\quad
\E[u_t^2 | \mathcal{F}_{t-1}] = \sigma^2(\bm{X}_{t-1};h),
\]
$\bm{X}_{t-1} = (X_{t-1},\ldots,X_{t-d})$, 
$\phi$ is a multi-dimensional functional,
$\{\bm{Z}_t\}$ is a high-dimensional covariate 
process, $\{\mathcal{F}_t\}$ is a filtration 
defined by 
$\mathcal{F}_t = \sigma(X_s, \bm{Z}_s, s \leq t)$,
and $h$ is a possibly infinite-dimensional nuisance parameter.
There are several previous works in which 
sparse estimation 
problems for time series models are considered.
For example, 
Lasso type estimators 
for a stationary Gaussian time series model in high-dimensional settings are considered in \cite{basu2015regularized}.
The Lasso for a stationary non-Gaussian time series model  is discussed  by \cite{wong2020lasso} in
high-dimensional settings.
In this paper, we consider the sparse estimation problems for conditionally heteroskedastic models and construct an 
asymptotically normal estimator for non-zero components of the true value of 
$\bm{\bm{\theta}} = (\bm{\bm{\alpha}}^\top, \bm{\bm{\beta}}^\top)^\top$ based on 
the theory of $Z$-estimation and variable selection.
As an example, we consider the higher order  
Integer-valued autoregressive model, discussed by, 
\cite{al1987first,du1991integer,latour1997multivariate, latour1998existence} and so on.
Moreover, we discuss an application to the 
estimation problems of Hawkes processes.

In Section \ref{sec:diffusion}, we consider models 
of diffusion processes with 
high-dimensional covariates.
There are some previous works in which 
sparse estimation problems for models 
of diffusion processes in high-dimensional 
settings are discussed.
In 
\cite{periera2014support,
gaiffas2019sparse,ciolek2020dantzig}, 
the 
(adaptive) Lasso or 
the Dantzig selector for drift parameters of 
linear models of diffusion processes 
such as Ornstein--Uhlenbeck processes under high-dimensional settings are studied.
Their methods are based on continuous observations, which do not deal with estimation problems of diffusion coefficients.
On the other hand, the 
Dantzig selector for drift parameters
of linear models of diffusion processes when the diffusion coefficient is a finite dimensional parameter based on 
discrete observations is discussed by \cite{fujimori2019dantzig}.
In this paper, we focus on  
seimiparametric models of diffusion processes
including the possibly infinite-dimensional parameter
$h$ in the diffusion coefficient, 
based on discrete observations and construct 
an asymptotically normal estimator for 
non-zero components of the true value of 
$\bm{\bm{\theta}} = (\bm{\bm{\alpha}}^\top, \bm{\bm{\beta}}^\top)^\top$.
As an example, we consider the 
high-dimensional Ornstein--Uhlenbeck processes 
as well as \cite{gaiffas2019sparse,
ciolek2020dantzig} and discuss a sparse estimation problem based on discrete observations.

There are three major contributions in this paper.
First, we provide sufficient conditions to derive the 
rate of convergence of the Dantzig selector 
for relatively general models including 
models of diffusion processes and time series under conditionally heteroskedasticity with possibly infinite-dimensional parameter. 
Second, we derive an asymptotically normal estimator by two step estimation for the high-dimensional parameter based on the theory of 
$Z$-estimators under sparse settings.
Since we consider the case where 
the dimension of the parameter diverges, 
we show the asymptotic normality of the 
two step estimator
in an infinite-dimensional Hilbert space,
which may be a new result.
Finally, we apply the general theory to the 
estimation problems for the 
integer autoregressive model of large order 
and the linear model of diffusion processes based on discrete observations.
We verify some sufficient conditions to derive the 
rate of convergence of the Dantzig selector,
such as some moment conditions and matrix conditions under some regularity conditions.


The rest of this paper is organized as follows.
In Section \ref{sec:general}, 
we explain our estimation procedure  
by using linear regression models.
We present a sufficient conditions to 
derive the rate of convergence and selection consistency of the Dantzig selector.
After the variable selection, we apply the general 
theory for $Z$-estimation with possibly infinite-dimensional nuisance parameter, which is originally
established by \cite{nishiyama2009asymptotic}.
See also \cite{vanwellner1996} for the asymptotic theory of $Z$-estimation.
The results for the ergodic time series are
provided in Section \ref{sec:time series}, which includes a 
concrete example for integer-valued autoregressive models and numerical simulations.
In Section \ref{sec:diffusion}, we present 
results for ergodic diffusion processes.
Some proofs and technical lemmas are 
provided in Section \ref{sec:proof}.
Especially, we use the maximal inequality 
provided in \cite{nishiyama2021martingale}
to verify one of the sufficient conditions
to derive the rate of convergence of 
the Dantzig selector in Sections \ref{sec:time series} and \ref{sec:diffusion}.

Throughout this paper, we denote by $\|\bm{v}\|_q$
the $l_q$ norm of a vector $\bm{v}$ and 
$\|\bm{A}\|_q$ the operator norm of a matrix $\bm{A}$
for every $q \geq 1$.
For a vector $\bm{v} \in \mathbb{R}^p$ and 
an index set $T \subset \{1,2,\ldots,p\}$, 
we denote by $\bm{v}_T$ the sub-vector of $\bm{v}$ restricted by $T$.
Moreover, $\bm{v}_{(T)}$ is a $p$-dimensional 
vector such that 
$\bm{v}_{(T)T} = \bm{v}_T$ and $\bm{v}_{(T)T^c} = \bm{0}$.
Similarly, for a $p \times p$ matrix $\bm{A}$ and index sets $T, T' \subset \{1,2,\ldots,p\}$, 
we denote by $\bm{A}_{T, T'} = (A_{ij})_{i \in T, j \in T'}$
the sub-matrix of $A$ restricted by 
$T$ and $T'$.
For a random variable $X$ and $q \geq1$, we write $\|X\|_{L^q}$
for the $L^q$-norm of $X$ if it exists;  
$\|X\|_{L^q} = \E[|X|^q]^{1/q}$.
For $d \in \N$, $\bm{\mu} \in \R^d$, and 
non-negative definite matrix 
$\bm{V} \in \R^{d \times d}$,
we write $\mathcal{N}_d(\bm{\mu}, \bm{V})$ for the 
$d$-dimensional normal distribution with 
mean $\bm{\mu}$ 
and covariance matrix $\bm{V}$.
Moreover, we assume that all random elements are measurable in this paper.

\section{Two step estimation methods via the Dantzig selector for linear regression models}\label{sec:general}
Let $(\Omega, \mathcal{F}, \Prob)$ be a probability space.
In this section, we consider the following model
\[
Y_t = \bm{\theta}^\top \bm{Z}_t + \epsilon_t,\quad
t=1,\ldots,n,
\]
where $Y_t$ is an $\mathbb{R}$-valued response, 
$\{\bm{Z}_t\}$ is an i.i.d. $\mathbb{R}^p$-valued covariate with a finite variance, 
$\{\epsilon_t\}$ is an independent random sequence
such that 
\[
\E[\epsilon_t | \bm{Z}_t] = 0,\quad
\E[\epsilon_t^2 | \bm{Z}_t] = \sigma^2(\bm{Z}_t ; h),
\]
with a measurable function $\sigma(\cdot; \cdot)$,
$\bm{\theta} \in \Theta \subset \mathbb{R}^p$ is an 
unknown parameter of interest,
 and 
$h \in H$ is a possibly infinite-dimensional nuisance parameter on a metric space $H$ 
equipped with a metric $d_H$.
Let $\bm{\theta}_0$ be the true value of $\bm{\theta}$.
We consider the 
following random maps;
\[
\psi_n^{(1)}(\bm{\theta}):=
\frac{1}{n}\sum_{t=1}^n \bm{Z}_t(Y_t -\bm{\theta}^\top \bm{Z}_t),
\quad
\Psi_n(\bm{\theta}, h):= \frac{1}{n}\sum_{t=1}^n \frac{\bm{Z}_t(Y_t -\bm{\theta}^\top \bm{Z}_t)}{\sigma^2(\bm{Z}_t; h)},
\]
and 
\[
\tilde{\Psi}_n(\bm{\theta}, h):= \frac{1}{n}\sum_{t=1}^n \frac{\bm{Z}_t \bm{Z}_t^\top(\bm{\theta}_0 -\bm{\theta})}{\sigma^2(\bm{Z}_t; h)}.
\]
The maps $\psi_n^{(1)}$ and $\Psi_n$ are score functions 
corresponding to least squares method and 
weighted least squares method, respectively.
Moreover, $\tilde{\Psi}_n$ is 
the compensater of $\Psi_n$ which satisfies that 
\[
\Psi_n(\bm{\theta}, h) - \tilde{\Psi}_n(\bm{\theta}, h)
= \frac{1}{n}\sum_{t=1}^n \frac{\bm{Z}_t \epsilon_t}{\sigma^2(\bm{Z}_t; h)}.
\]
We are interested in the estimation problem for  
the true value $\bm{\bm{\theta}}_0=({\theta}_{01},\ldots,{\theta}_{0p})^\top$ under the following high-dimensional and sparse settings.
\begin{itemize}
\item[(i)]
The dimension $p$ of $\bm{\bm{\theta}}$ possibly tends to
$\infty$ as $n \to \infty$.
\item[(ii)]
Let $T_0$ be the support index set of $\bm{\bm{\theta}}_0$, 
i.e., 
$T_0 = \{j : {\theta}_{0j} \not =0\}$.
The sparsity $s$, which is the cardinality of 
$T_0$, is
smaller than $n$ and $p$.
\end{itemize}
We first construct an estimator
for $T_0$ to achieve the dimension reduction.
To do this, we consider the following 
Dantzig selector type estimator 
$\hat{\bm{\bm{\theta}}}_n^{(1)}$ for $\bm{\bm{\theta}}_0$ 
based on $\psi_n^{(1)}$:
\[
\hat{\bm{\bm{\theta}}}_n^{(1)}
:= \arg \min_{\bm{\bm{\theta}} \in \mathcal{C}_n} \|\bm{\bm{\theta}}\|_1,\quad
\mathcal{C}_n := \{\bm{\bm{\theta}} \in {\Theta} : \|\psi_n^{(1)}(\bm{\bm{\theta}})\|_\infty \leq \lambda_n\},
\]
where $\lambda_n$ is a tuning parameter.
Using $\hat{\bm{\bm{\theta}}}_n^{(1)}$, we define 
$\hat{T}_n$ as follows:
\[
\hat{T}_n := \{j : |\hat{{\theta}}_{nj}^{(1)}| > \tau_n\},
\]
where $\tau_n$ is a threshold, which is a 
tuning parameter possibly depending on $n, p$, and $s$.
For every index set $T \subset \{1, \ldots,p\}$, let 
$C_T$ be a set defined by 
\[
C_T := \{\bm{v} \in \mathbb{R}^p: \|\bm{v}_{T^c}\|_1 \leq \|\bm{v}_T\|_1 \}.
\]
The $p \times p$ matrix $\bm{V}_n^{(1)}$ 
satisfies that 
for every $\bm{v} \in C_{T_0}$ and $\bm{\bm{\theta}} \in {\Theta}$,
\begin{equation}\label{psi-V}
|\bm{v}^\top [\psi_n^{(1)}(\bm{\bm{\theta}}) - \psi_n^{(1)}(\bm{\bm{\theta}}_0)]|
\geq |\bm{v}^\top \bm{V}_n^{(1)}\bm{v}|.
\end{equation}
For a matrix $\bm{M} \in \mathbb{R}^{p \times p}$,
we define the following quantity as well as \cite{huang2013oracle}:
\begin{equation}\label{def: F_infty}
F_\infty(T, \bm{M})
:= \inf_{\bm{v} \in C_{T_0}\setminus\{ \bm{0}\} } \frac{|\bm{v}^\top \bm{M} \bm{v}|}{\|\bm{v}_{T}\|_1\|\bm{v}\|_\infty}.
\end{equation}
See also \cite{buhlmann2011statistics}
for such matrix factors.
We assume the following conditions.
\begin{assumption}\label{DS condition}
\begin{itemize}
\item[(i)]
It holds that 
\[
\Prob\left(
\|\psi_n^{(1)}(\bm{\bm{\theta}}_0)\|_\infty > {\lambda}_n
\right) \to 0,\quad
n \to \infty.
\]
\item[(ii)]
There exists a positive constant $\delta$ such that 
\[
\Prob(F_\infty(T_0, \bm{V}_n^{(1)})> \delta) \to 1,\quad
n \to \infty.
\]
\item[(iii)]
Let $c_n := \|\psi_n^{(1)}(\hat{\bm{\theta}}_n^{(1)})-\psi_n^{(1)}(\bm{\theta}_0)\|_\infty$
and $\theta_{\min} := \inf_{j \in T_0} |\theta_{0j}|$.
Then, it holds that 
\[
\theta_{\min} > 4c_n/\delta.
\]
\end{itemize}
\end{assumption}
If $\psi_n^{(1)}(\bm{\bm{\theta}}_0)$ is a 
terminal value of square integrable martingale,
we can verify the condition (i) by using, e.g., 
stochastic maximal inequality provided in 
\cite{nishiyama2021martingale} under some moment conditions on the covariate $\{\bm{Z}_t\}$ and the noise $\{\epsilon_t\}$.
The condition (ii) strongly 
depends on the model assumption such as 
stationarity, tail property, dependence structure 
of the processes, and the 
regularity condition of the information matrix.
Therefore, in Sections 3 and 4, we check the condition (i)
for relatively general models of 
stochastic processes, while the
condition (ii) is verified for more concrete 
models such as stationary 
integer-valued autoregressive models and Ornstein--Uhlenbeck processes.
The condition (iii) is used to verify the 
performance of $\hat{T}_n$.
Note that since it holds that $c_n \to^p 0$ as $n \to \infty$ under regularity conditions 
for some concrete models in Sections 3 and 4, 
the lower bound of $\theta_{\min}$ can be sufficiently small as far as $\delta$ is 
bounded from below by a positive constant for sufficiently large $n$.

Under Assumption \ref{DS condition}, we establish the following theorem.
\begin{thm}\label{rate of convergence DS general model}
Let Assumption \ref{DS condition} hold.
Then, 
it holds that
\begin{equation}\label{error bound general}
\Prob\left(\|\hat{\bm{\bm{\theta}}}_n^{(1)} - \bm{\bm{\theta}}_0\|_\infty > \frac{2}{\delta}c_n\right)
\to 0,\quad
n \to \infty,
\end{equation}
where $c_n := \|\psi_n^{(1)}(\hat{\bm{\theta}}_n^{(1)})-\psi_n^{(1)}(\bm{\theta}_0)\|_\infty$.
Especially, it holds that 
\[
\|\hat{\bm{\bm{\theta}}}_n^{(1)} - \bm{\bm{\theta}}_0\|_\infty = O_p(\lambda_n),\quad n \to \infty.
\]
Moreover, 
if the threshold 
$\tau_n$ satisfies that 
$2 c_n / \delta< \tau_n < \theta_{\min}/2$,
then
it holds that
\begin{equation}\label{selection general}
\Prob(\hat{T}_n = T_0) \to 1,\quad n \to \infty.
\end{equation}
\end{thm}
\begin{proof}
For the proof of \eqref{error bound general}, 
it suffices to show that 
\[
\|\hat{\bm{\bm{\theta}}}_n^{(1)}-\bm{\bm{\theta}}_0\|_\infty 
\leq \frac{2}{\delta}c_n
\]
under the event 
\[
\{\|\psi_n^{(1)}(\bm{\bm{\theta}}_0)\|_\infty \leq {\lambda}_n\}\cap\left\{F_\infty(T_0, \bm{V}_n^{(1)})>\delta\right\}.
\]
Put $\bm{v} = \hat{\bm{\bm{\theta}}}_n^{(1)}-\bm{\theta}_0$.
Since $\hat{\bm{\theta}}_n^{(1)}=\bm{\theta}_0 + \bm{v}$ is a minimizer of $l_1$ norm, we can show that
\begin{eqnarray*}
0 \geq \|\bm{\theta}_0 + \bm{v}\|_1 - \|\bm{\theta}_0\|_1 
&=& \sum \limits_{j \in T_0^c} |\bm{v}_{T^c_{0}j}| + \sum \limits_{j \in T_0}
 (|\bm{\theta}_{0j}+\bm{v}_{T_{0}j}| - |\bm{\theta}_{0j}|)\\
 &\geq& \sum \limits_{j \in T_0^c} |\bm{v}_{T^c_{0}j}| + \sum \limits_{j \in T_0}
 (|\bm{\theta}_{0j}|-|\bm{v}_{T_{0}j}| - |\bm{\theta}_{0j}|)\\
 &\geq& \sum \limits_{j \in T_0^c} |\bm{v}_{T_0^c j}| -  \sum \limits_{j \in T_0} |\bm{v}_{T_{0}j}| \\
 &=& \|\bm{v}_{T_0^c}\|_1 - \|\bm{v}_{T_0}\|_1,
\end{eqnarray*} 
which means
$\bm{v} \in C_{T_0}$.
Moreover, we have 
\[
\|\bm{v}\|_1 = \|\bm{v}_{T_0^c}\|_1 + \|\bm{v}_{T_0}\|_1
\leq 2 \|\bm{v}_{T_0}\|_1.
\]
Noting that $c_n := \|\psi_n^{(1)}(\hat{\bm{\theta}}_n^{(1)})-\psi_n^{(1)}(\bm{\theta}_0)\|_\infty$, we have
\begin{eqnarray*}
|\bm{v}^\top [\psi_n^{(1)}(\hat{\bm{\theta}}_n^{(1)})-\psi_n(\bm{\theta}_0)]|
\leq \|\bm{v}\|_1c_n.
\end{eqnarray*}
Furthermore, it follows from the definitions of 
$\bm{V}_n^{(1)}$ and 
$F_\infty(T_0, \bm{V}_n^{(1)})$ that 
\[
|\bm{v}^\top [\psi_n^{(1)}(\hat{\bm{\theta}}_n^{(1)})-\psi_n(\bm{\theta}_0)]|
\geq |\bm{v}^\top \bm{V}_n^{(1)}\bm{v}|
\geq 
F_\infty(T_0, \bm{V}_n^{(1)}) \|\bm{v}_{T_0}\|_1 \|\bm{v}\|_\infty
> \delta \|\bm{v}_{T_0}\|_1 \|\bm{v}\|_\infty.
\]
Combining these facts, we have 
\[
\delta \|\bm{v}_{T_0}\|_1 \| \bm{v} \|_\infty
\leq \|\bm{v}\|_1c_n \leq 2\|\bm{v}_{T_0}\|_1c_n.
\]
It implies 
\[
\|\bm{v}\|_\infty
\leq \frac{2}{\delta}c_n,
\]
which concludes the proof of \eqref{error bound general}.
Moreover, by using the triangle inequality 
and the definition of $\hat{\bm{\theta}}_n^{(1)}$,
we have 
\[
c_n = \|\psi_n^{(1)}(\hat{\bm{\theta}}_n^{(1)})-\psi_n^{(1)}(\bm{\theta}_0)\|_\infty
\leq 2 \lambda_n,
\]
which implies 
$\|\bm{v}\|_\infty = O_p(\lambda_n)$.

To prove \eqref{selection general}, it suffices to show that 
\[
\{\|\hat{\bm{\theta}}_n^{(1)} - \bm{\theta}_0\|_\infty \leq 2c_n/\delta\} \subset \{\hat{T}_n = T_0\}.
\]
Therefore, we consider the following two cases 
under the event $\{\|\hat{\bm{\theta}}_n^{(1)} - \bm{\theta}_0\|_\infty \leq 2c_n/\delta\}$.
\begin{itemize}
\item[(i)]
Suppose that $j \in T_0$.
Then, it holds that
\[
|{\theta}_{0j}| - |\hat{{\theta}}^{(1)}_{nj}|
\leq |\hat{{\theta}}^{(1)}_{nj}-{\theta}_{0j}| \leq 2c_n/\delta
\]
under the event $\{\|\hat{{\bm{\theta}}}_n^{(1)} - \bm{\theta}_0\|_\infty \leq 2c_n/\delta\}$.
Since $\tau_n < \inf_{j \in T_0}|{\theta}_{0j}| - 2c_n/\delta$,
we have 
\[
|\hat{{\theta}}_{nj}| \geq |{\theta}_{0j}| - 2c_n/\delta
> \tau_n,
\]
which implies $j \in \hat{T}_n$, i.e., 
$T_0 \subset \hat{T}_n$.
\item[(ii)]
Suppose that $j \in T_0^c$, then
it holds that 
\[
|\hat{{\theta}}^{(1)}_{nj}-{\theta}_{0j}| = |\hat{{\theta}}^{(1)}_{nj}| \leq 2c_n/\delta < \tau_n,
\]
which implies $j \in \hat{T}_n^c$, i.e., 
$T_0^c \subset \hat{T}_n^c$.
\end{itemize}
It follows from (i) and (ii) that 
$\hat{T}_n = T_0$.
\end{proof}
\begin{rem}\label{DS consistency general model}
If we choose the tuning parameter 
$\lambda_n$ satisfying that
$\lambda_n \to 0$ as $n \to \infty$,
then $\hat{\bm{\theta}}_n^{(1)}$
is a consistent estimator in terms of $l_\infty$ norms.
For the purpose of deriving the 
asymptotic behavior of $\hat{T}_n$, it is sufficient 
to obtain the $l_\infty$ error.
We can also obtain the inequality for the 
$l_q$ errors for $q \geq 1$ under some appropriate 
conditions for the Hessian matrix $V_n^{(1)}$.
\end{rem}
\begin{rem}\label{choice of threshold}
We can choose $\tau_n$ from the 
following interval 
$ (2 c_n/\delta, \theta_{\min}/2)$ to verify 
the variable selection consistency \eqref{selection general}.
Since the lower bound $2 c_n/\delta$ of the above 
interval converges to zero as 
$n \to \infty$, we can choose 
the threshold $\tau_n$ small enough
when $n$ is sufficiently large.
\end{rem}
Next, assume that we can construct a consistent 
estimator of $h$.
This can be achieved under some appropriate conditions.
We provide the following examples.
\begin{eg}\label{ex1 variance}
Suppose that $H \subset \mathbb{R}^{p_H}$, where $p_H$ is 
independent of $n$ and $p$, 
and $\sigma^2(\cdot; \cdot)$ is a 
known function up to 
the parameter $\bm{h}$.
Then, we can construct the least squares estimator 
as follows:
\[
\hat{\bm{h}}_n
= \arg \min_{\bm{h} \in H} \mathcal{B}_n(\bm{h}),
\]
where 
\[
\mathcal{B}_n(\bm{h})
= \frac{1}{n}\sum_{t=1}^n \left\{
(Y_t-\hat{\bm{\theta}}_n^{(1) \top}\bm{Z}_t)^2
- \sigma^2(\bm{Z}_t; \bm{h})
\right\}^2.
\]
Sufficient conditions for the consistency 
of $\hat{\bm{h}}_n$ can be found in 
\cite{nishiyama2009asymptotic}.
\end{eg}
\begin{eg}\label{ex2 variance}
Suppose that $H \subset \mathbb{R}^p$ and 
$\sigma^2(\bm{Z}_t, \bm{h}) = \sigma^2(\bm{h}^\top \bm{Z}_t)$, where $\sigma^2(\cdot)$ is a known function. 
\begin{itemize}
\item[(i)]
When 
$T_{h} := \{j : h_{0j} \not= 0\} \subset T_0$,
we can construct the estimator as follows:
\[
\hat{\bm{h}}_{n \hat{T}_n}
= \arg \min_{\bm{h} \in H} \mathcal{B}_n(\bm{h}_{\hat{T}_n}),\quad
\hat{\bm{h}}_{n \hat{T}_n^c} = \bm{0},
\]
where 
\[
\mathcal{B}_n(\bm{h}_{\hat{T}_n})
= \frac{1}{n}\sum_{t=1}^n \left\{
(Y_t-\hat{\bm{\theta}}_{n \hat{T}_n}^{(1) \top}\bm{Z}_{
t \hat{T}_n})^2
- \sigma^2(\bm{h}_{\hat{T}_n}^\top \bm{Z}_{t \hat{T}_n})
\right\}^2.
\]
Under the assumptions of Theorem \ref{rate of convergence DS general model}, the consistency can be proved by the combination of  
Slutsky theorem and the similar way to prove the 
consistency of the estimator in Example \ref{ex1 variance}.
\item[(ii)]
When $T_h \not \subset T_0$, 
we should apply some sparse estimation method to 
estimate $\bm{h}$.
For example, consider the following score 
function:
\[
\Phi_n(\bm{h})
= \frac{1}{n} \sum_{t=1}^n \frac{\partial}{\partial \bm{h}} \sigma^2(\bm{h}^\top \bm{Z}_t)
\left\{
(Y_t-\hat{\bm{\theta}}_n^{(1) \top}\bm{Z}_t)^2
- \sigma^2(\bm{Z}_t; \bm{h})
\right\}.
\]
Then, applying the 
Dantzig selector type estimation procedure, 
we obtain a consistent estimator 
under appropriate conditions.
\end{itemize}
\end{eg}
\begin{eg}\label{ex3 variance}
Suppose that 
$H$ is a set of continuous functions
and 
$\sigma^2(\bm{Z}_t, h) = h(\bm{Z}_t)$ with unknown 
function $h$.
Consider the case that 
$h(\bm{Z}_t) = h({\bm{Z}}_{t(T_0)})$, 
where for a vector $\bm{z} \in \mathbb{R}^p$ and $T$, 
${\bm{z}}_{(T)}$ is a 
$p$-dimensional vector such that 
${\bm{z}}_{(T)T} = \bm{z}_T$ and 
${\bm{z}}_{(T)T^c} = \bm{0}$.
This means that 
the conditional variance of the error term 
depends only on $\bm{Z}_{t T_0}$.
Then, we can consider some nonparametric estimation method such as  
kernel estimation by using $\bm{Z}_{t \hat{T}_n}, t=1,\ldots,n$.
See, e.g., \cite{masry1995nonparametric} or 
\cite{carrol1982adapting} for 
sufficient conditions to prove the consistency 
of the kernel estimator.
\end{eg}
Hereafter, to introduce an asymptotically normal estimator for non-zero components of $\bm{\theta}_0$, we assume that $s$ is independent of $n$ and $p$.
For every index set $T \subset \{1,\ldots,p\}$, 
we consider the following random maps 
on $\Theta_T \times H$ to $\mathbb{R}^{|T|}$:
\[
\Psi_{n T} : {\Theta}_T \times H \to \mathbb{R}^{|T|},\quad
\tilde\Psi_{n T} : {\Theta}_T \times H \to \mathbb{R}^{|T|},
\]
where ${\Theta}_T$ is the set of sub-vectors
of ${\Theta}$ restricted by $T$ and 
\[
\Psi_{n T} (\bm{\theta}_T, h)=
\frac{1}{n} \sum_{t=1}^n  \frac{\bm{Z}_{t T}(Y_t- \bm{\theta}_{T}^\top \bm{Z}_{t T})}{\sigma^2({\bm{Z}}_{t(T)}; h)},\quad
\tilde{\Psi}_{n T} (\bm{\theta}_T, h)
= \frac{1}{n} \sum_{t=1}^n  \frac{\bm{Z}_{t T} \bm{Z}_{t T}^\top (\bm{\theta}_{0T}-\bm{\theta}_{T})}{\sigma^2({\bm{Z}}_{t(T)}; h)}.
\]
Let $\hat{h}_n$ be an estimator of $h \in H$ such that $d_H(\hat{h}_n, h_0) = o_p(1)$ as $n \to \infty$.
Then, we can construct the new estimator 
$\tilde{\bm{\theta}}_n$ for the parameter $\bm{\theta}$ of interest, with help of $\hat{h}_n$ and 
$\hat{T}_n$, as a solution to the following 
equations:
\[
\Psi_{n\hat{T}_n}(\tilde{\bm{\theta}}_{n\hat{T}_n}, \hat{h}_n) \approx \bm{0},\quad
\tilde{\bm{\theta}}_{n\hat{T}_n^c} = \bm{0}.
\]
We consider the following conditions to 
deduce the asymptotic behavior 
of $\tilde{\bm{\theta}}_n$.
\begin{assumption}\label{regularity sec2}
\begin{itemize}
\item[(i)]
The function $\sigma^2$ satisfies that
\[
\sigma^2(\bm{z}; h)
= \sigma^2(\bm{z}_{(T_0)}; h),\quad
\bm{z} \in \mathbb{R}^p.
\]
\item[(ii)]
There exists a $\mathbb{R}$-valued 
measurable function 
$\Lambda(\cdot)$ on $\mathbb{R}^{s}$ such that 
\[
\sup_{\bm{\theta}_{T_0} \in \Theta_{T_0}} 
|\bm{\theta}_{T_0}^\top \bm{z}_{T_0}| \leq \Lambda(\bm{z}_{T_0}),
\]
\[
\sup_{h \in H} \sigma^2(\bm{z}; h) \leq \Lambda(\bm{z}_{T_0}),
\]
and for every $h_1, h_2 \in H$, 
\[
|\sigma^2(\bm{z}_{(T_0)}; h_1) - \sigma^2(\bm{z}_{(T_0)}; h_2)|
\leq \Lambda(\bm{z}_{T_0})d_H(h_1, h_2).
\]
\item[(ii)]
There exists a constant $c>0$ such that 
\[
\inf_{\bm{z}_{T_0} \in \mathbb{R}^s, h \in H} \sigma^2(\bm{z}_{(T_0)}; h) >c.
\]
\item[(iii)]
The matrix 
\[
\bm{I}(\bm{\theta}_{0 T_0}, h_0)
= \E \left[
\frac{\bm{Z}_{t T_0} \bm{Z}_{t T_0}^\top}{\sigma^2(\bm{Z}_{t (T_0)}; h_0)}
\right]
\]
is finite and invertible.
\item[(iv)]
It follows that 
\[
\int_0^1 \sqrt{\log N(H, d_H, \epsilon)} d \epsilon < \infty,
\]
where $N(H, d_H, \epsilon)$ is 
$\epsilon$-covering number of $H$ with respect to 
the metric $d_H$.
\item[(v)]
It holds that 
\[
\E[|\Lambda(\bm{Z}_{t T_0})|^4] < \infty.
\]
\end{itemize}
\end{assumption} 
Then, we have the following lemma.
\begin{lem}\label{jain-marcus}
Suppose that Assumptions \ref{DS condition}
and \ref{regularity sec2} hold and 
let $\rho$ be a metric in $\Theta_{T_0} \times H$
defined by 
\[
\rho((\bm{\theta}_{T_0}, h), (\bm{\theta}_{T_0}', h'))
= \|\bm{\theta}_{T_0}-\bm{\theta}_{T_0}'\|_2 \lor  
d_H(h, h').
\]
\begin{itemize}
\item[(i)]
It holds that 
\[
\sqrt{n} \{\Psi_{n T_0} - \tilde{\Psi}_{n T_0}\}
\to^d G_{T_0}\quad n \to \infty
\]
in $C_{\rho}(\Theta_{T_0}\times H)$,
where 
$G_{T_0}$ is a zero-mean Gaussian random field 
with covariance function
\[
\E\left[
G_{T_0}(\bm{\theta}_{T_0}, h)G_{T_0}(\bm{\theta}_{T_0}, h)^\top
\right]
= \E \left[
\frac{\bm{Z}_{t T_0}\bm{Z}_{t T_0}^\top \sigma^2(\bm{Z}_{t (T_0)}; h_0)}{\sigma^2(\bm{Z}_{t(T_0)}; h)\sigma^2(\bm{Z}_{t(T_0)}; h)}
\right],
\]
and 
\[
C_{\rho}(\Theta_{T_0}\times H)
= \{f : \Theta_{T_0}\times H \to \mathbb{R}\ |\  
f\ \mbox{is continuous with respect to }\rho\}.
\]
\item[(ii)]
For every sequence $\{(\bm{\theta}_{n T_0}, h_n)\}$
such that 
$\rho((\bm{\theta}_{n T_0}, h_n), (\bm{\theta}_{0 T_0}, h_0)) = o_p(1)$ as $n \to \infty$, 
it holds that 
\[
\tilde{\Psi}_{n T_0}({\bm{\theta}}_{n T_0}, {h}_n)
- \tilde{\Psi}_{n T_0}(\bm{\theta}_{T_0}, h_0)
-\bm{I}(\bm{\theta}_{0 T_0}, h_0)({\bm{\theta}}_{n T_0}-\bm{\theta}_{0 T_0})
= o_{p}(\|{\bm{\theta}}_{n T_0}-\bm{\theta}_{0 T_0}\|_2)
\]
\end{itemize}
\end{lem}
\begin{proof}
\begin{itemize}
\item[(i)]
Note that 
\[
\sqrt{n} \{\Psi_{n T_0}(\bm{\theta}_{T_0}, h) - \tilde{\Psi}_{n T_0}(\bm{\theta}_{T_0}, h)\}
= \frac{1}{\sqrt{n}}\sum_{t=1}^n \frac{\bm{Z}_{t T_0} \epsilon_t}{\sigma^2(\bm{Z}_{t(T_0)}; h)}.
\]
Then, we can apply Proposition 4.5 of \cite{nishiyama2000weak} to deduce the fact.
\item[(ii)]
Using the law of large numbers and 
Assumption \ref{regularity sec2},
we have 
\begin{eqnarray*}
\lefteqn{
\tilde{\Psi}_{n T_0}({\bm{\theta}}_{n T_0}, {h}_n)
- \tilde{\Psi}_{n T_0}(\bm{\theta}_{0T_0}, h_0)}\\
&=& \frac{1}{n} \sum_{t=1}^n \frac{\bm{Z}_{t T_0} \bm{Z}_{t T_0}^\top}{\sigma^2(\bm{Z}_{t(T_0)}; h_n)}(\bm{\theta}_{0 T_0}- \bm{\theta}_{n T_0})\\
&=& \frac{1}{n} \sum_{t=1}^n \frac{\bm{Z}_{t T_0} \bm{Z}_{t T_0}^\top}{\sigma^2(\bm{Z}_{t (T_0)}; h_0)}(\bm{\theta}_{0 T_0}- \bm{\theta}_{n T_0})\\
&&+ \frac{1}{n} \sum_{t=1}^n \bm{Z}_{t T_0} \bm{Z}_{t T_0}^\top (\bm{\theta}_{0 T_0}- \bm{\theta}_{n T_0})
\left\{
\frac{1}{\sigma^2(\bm{Z}_{t(T_0)}; h_n)} - \frac{1}{\sigma^2(\bm{Z}_{t(T_0)}; h_0)}
\right\}\\
&=& - \bm{I}(\bm{\theta}_{0 T_0}, h_0)(\bm{\theta}_{n T_0}-\bm{\theta}_{0 T_0})
+ o_p(\|\bm{\theta}_{n T_0}-\bm{\theta}_{0 T_0}\|_2).
\end{eqnarray*}
\end{itemize}
\end{proof}
Then, we establish the asymptotic normality of $\tilde{\bm{\theta}}_n$.
\begin{lem}\label{normality general model}
Suppose that Assumptions \ref{DS condition}
and \ref{regularity sec2} hold.
Assume moreover that 
$2 c_n / \delta< \tau_n < \theta_{\min}/2$.
Then, it holds that 
\[
\sqrt{n} (\tilde{\bm{\theta}}_{n \hat{T}_n}- \bm{\theta}_{0 T_0})1_{\{\hat{T}_n = T_0\}}
\to^d \mathcal{N}_s(\bm{0}, \bm{I}(\bm{\theta}_{0 T_0}, h_0)^{-1}),\quad n \to \infty.
\]
\end{lem}
\begin{proof}
Note that 
\[
\sqrt{n}(\tilde{\bm{\theta}}_{n \hat{T}_n} - \bm{\theta}_{0 T_0})1_{\{\hat{T}_n = T_0\}}
=\sqrt{n}(\hat{\bm{\theta}}_{n T_0} - \bm{\theta}_{0 T_0})1_{\{\hat{T}_n = T_0\}}.
\]
and that 
\[
1_{\{\hat{T}_n = T_0\}} \to^p 1,\quad n \to \infty
\]
by Theorem \ref{rate of convergence DS general model}.
It follows from Theorem 2.1 of \cite{nishiyama2009asymptotic} that 
\[
\sqrt{n}(\tilde{\bm{\theta}}_{n T_0} - \bm{\theta}_{0 T_0})
\to^d -\bm{I}(\bm{\theta}_0, h_0)^{-1}G_{T_0}(\bm{\theta}_{0 T_0}, h_0),\quad
n \to \infty.
\]
Then, using the Slutsky lemma, we obtain the conclusion. 
\end{proof}
Let $\ell^2$ be a Hilbert space defined by 
$\ell^2 :=\{\bm{x} \in \mathbb{R}^\infty : \sum_{j=1}^\infty |x_j|^2 < \infty\}$ with the inner product 
$\langle \bm{x}, \bm{y} \rangle := \sum_{j=1}^\infty x_j y_j$.
For $\bm{u} \in \ell^2$ and an finite index set $T$, 
we denote by $\bm{u}_T \in \mathbb{R}^{|T|}$ 
the sub-vector of $\bm{u}$ restricted by $T$.
Consider the $\ell^2$-valued random sequence 
$\{\bm{\mathcal{R}}_n\}_{n \in \mathbb{N}}$ defined by 
\[
\langle \bm{e}_j, \bm{\mathcal{R}}_n\rangle
= \begin{cases}
\sqrt{n} (\tilde{\bm{\theta}}_{nj} - \bm{\theta}_{0j} ) & 
1 \leq j \leq p \\
0 & j > p
\end{cases},
\] 
where $\{\bm{e}_j\}_{j \in \mathbb{N}}$ is the 
canonical basis of $\mathbb{R}^\infty$.
\begin{lem}\label{cramer-wold general}
Under the same assumptions as Lemma \ref{normality general model}, 
it holds that 
\[
\langle \bm{u}, \bm{\mathcal{R}_n} \rangle
\to^d \mathcal{N}(0, \bm{u}_{T_0}^\top \bm{I}(\bm{\theta}_{0 T_0}, h_0)^{-1} \bm{u}_{T_0}),\quad
n \to \infty,
\]
for every $\bm{u} \in \ell^2$.
\end{lem}
\begin{proof}
Note that 
\begin{eqnarray*}
\langle \bm{u}, \bm{\mathcal{R}_n} \rangle
&=& \sqrt{n} \bm{u}_{\hat{T}_n}^\top(\tilde{\bm{\theta}}_n - \bm{\theta}_0)_{\hat{T}_n}
+ \sqrt{n} \bm{u}_{\hat{T}_n^c}^\top(\tilde{\bm{\theta}}_n - \bm{\theta}_0)_{\hat{T}_n^c}.
\end{eqnarray*}
As for the second term, we have 
\begin{eqnarray*}
\sqrt{n} \bm{u}_{\hat{T}_n^c}^\top(\tilde{\bm{\theta}}_n - \bm{\theta}_0)_{\hat{T}_n^c}
&=& - \sqrt{n} \bm{u}_{\hat{T}_n^c}^\top \bm{\theta}_{0 \hat{T}_n^c} \\
&=& - \sqrt{n} \bm{u}_{{T}_0^c}^\top \bm{\theta}_{0 \hat{T}_0^c} 1_{\{\hat{T}_n = T_0\}}
 - \sqrt{n} \bm{u}_{\hat{T}_n^c}^\top \bm{\theta}_{0 \hat{T}_n^c} 1_{\{\hat{T}_n \not= T_0\}}\\
&=& 0 - \sqrt{n} \bm{u}_{\hat{T}_n^c}^\top \bm{\theta}_{0 \hat{T}_n^c} 1_{\{\hat{T}_n \not= T_0\}}.
\end{eqnarray*}
It follows from \eqref{selection general} that 
for every $\eta >0$ and sufficiently large $n$ such that $\sqrt{n}>\eta$,
\begin{eqnarray*}
\Prob(\sqrt{n} 1_{\{\hat{T}_n \not= T_0\}} > \eta)
&=& \Prob(1_{\{\hat{T}_n \not= T_0\}} = 1) \\
&=& \Prob(\hat{T}_n \not= T_0) \to 0,\quad
n \to \infty,
\end{eqnarray*}
which implies that 
\[
\sqrt{n} \bm{u}_{\hat{T}_n^c}^\top(\tilde{\bm{\theta}}_n - \bm{\theta}_0)_{\hat{T}_n^c} = o_p(1),\quad 
n \to \infty.
\]
By the similar argument, we have 
\begin{eqnarray*}
\langle \bm{u}, \bm{\mathcal{R}_n} \rangle
&=& \sqrt{n} \bm{u}_{\hat{T}_n}^\top(\tilde{\bm{\theta}}_n - \bm{\theta}_0)_{\hat{T}_n} + o_p(1) \\
&=& \sqrt{n} \bm{u}_{T_0}^\top(\tilde{\bm{\theta}}_n - \bm{\theta}_0)_{T_0} 1_{\{\hat{T}_n = T_0\}}
+\sqrt{n} \bm{u}_{\hat{T}_n}^\top(\tilde{\bm{\theta}}_n - \bm{\theta}_0)_{\hat{T}_n} 1_{\{\hat{T}_n \not= T_0\}} + o_p(1) \\
&=& \sqrt{n} \bm{u}_{T_0}^\top(\tilde{\bm{\theta}}_n - \bm{\theta}_0)_{T_0} 1_{\{\hat{T}_n = T_0\}} + o_p(1).
\end{eqnarray*}
Then, we have the conclusion from 
Lemma \ref{normality general model}.
\end{proof}
Then, we have the following 
limit theorem for $\{\bm{\mathcal{R}}_n\}_{n \in \mathbb{N}}$.
\begin{thm}\label{weak convergence l2 general}
Under the same assumptions as Lemma \ref{normality general model}, it holds that 
\[
\bm{\mathcal{R}}_n \to^d \bm{\mathcal{R}},\quad n \to \infty \quad
\mbox{in}\ \ell^2,
\]
where 
$\bm{\mathcal{R}}$ is a centered Gaussian random field whose distribution is determined by 
the distribution of its marginals 
\[
\langle \bm{u}, \bm{\mathcal{R}} \rangle \sim
\mathcal{N}(0, \bm{u}_{T_0}^\top\bm{I}(\bm{\theta}_0, h_0)\bm{u}_{T_0})
\]
for every $\bm{u} \in \ell^2$.
\end{thm}
\begin{proof}
For every $\delta>0$, it holds that
\begin{eqnarray*}
\Prob\left(
\|\sqrt{n} (\tilde{\bm{\theta}}_{n}- \bm{\theta}_0)_{T_0^c}\|_2^2 > \delta
\right)
&\leq& \Prob\left(
\{n \|\tilde{\theta}_{n T_0^c}\|_2^2 > \delta 
\cap \{\hat{T}_n = T_0^c\}
\right)\\
&&+ \Prob\left(
\{n \|\tilde{\theta}_{n T_0^c}\|_2^2 > \delta 
\cap \{\hat{T}_n \neq T_0^c\}
\right)\\
&\leq& \Prob\left(
\{n \|\tilde{\theta}_{n \hat{T}_n^c}\|_2^2 > \delta 
\cap \{\hat{T}_n = T_0^c\}
\right)
+ \Prob(\hat{T}_n \neq T_0)\\
&=& 0 + o(1),\quad n \to \infty.
\end{eqnarray*}
Therefore, $\{\bm{\mathcal{R}}_n\}_{n \in \mathbb{N}}$
is asymptotically finite-dimensional; see, e.g., Section 1.8 of \cite{vanwellner1996}. 

It follows from Lemma \ref{cramer-wold general} that
for every $\bm{u} \in \ell^2$, 
\[
\langle \bm{u}, \bm{\mathcal{R}}_n \rangle
\to^d \langle \bm{u}, \bm{\mathcal{R}} \rangle,\quad
n \to \infty.
\]
Then, Theorem 1.8.4 of \cite{vanwellner1996}
deduce the conclusion.
\end{proof}
\begin{rem}
\begin{itemize}
\item[(i)]
In Theorem \ref{rate of convergence DS general model}, we establish the 
rate of convergence and the 
selection consistency of the 
Dantzig selector when $s$ is allowed to diverge 
under appropriate conditions.
On the other hand, to apply Theorem 
\ref{normality general model} directly, 
the sparsity $s$ should be fixed.
To consider the asymptotic normality in the case  where $s$ can be diverge, 
we should consider another type of 
asymptotic theory.
See, e.g., \cite{chernozhukov2021lasso} for the 
Gaussian approximation for 
time series models in such a high-dimensional regime.
\item[(ii)]
Theorem \ref{weak convergence l2 general}
is an extension of the 
asymptotic normality of the estimator 
in finite-dimensional cases.
Since the asymptotic variance achieves the 
optimal bound, the two-step 
estimator satisfies the 
optimality when the sparsity is fixed.
\end{itemize}
\end{rem}
The results in this section 
can be naturally generalized to 
various models including time series and stochastic processes.
For such models, the random maps $\psi_n^{(1)}$, 
$\Psi_n$ can be defined as a 
score functions based on (conditional) least squares method and 
quasi-likelihood method or weighted least squares methods.
Then, if $\psi_n^{(1)}$ and the corresponding 
Hessian matrix satisfy Assumption \ref{DS condition}, 
we can derive the 
rate of convergence of
the Dantzig selector type estimator 
$\hat{\theta}_n^{(1)}$ 
and the performance of the 
estimator $\hat{T}_n$ 
as well as Theorem \ref{rate of convergence DS general model}.
Moreover, using the selection consistency,  
the score function $\Psi_n$, and 
its compensator $\tilde{\Psi}_n$, 
we can derive the 
asymptotic normality of the 
two-step estimator 
as well as Theorem \ref{normality general model}
under appropriate conditions.
In subsequent sections, we discuss the 
application to the time series models and 
the models of the diffusion processes.
\section{Applications to time series models}\label{sec:time series}
\subsection{Models of ergodic time series with high-dimensional parameters}
Let us consider the following time series model:
\[
X_t = \tilde{S}(\bm{\alpha}^\top \tilde{\phi}(X_{t-1},\ldots,X_{t-d_1}), \bm{\beta}^\top \bm{Z}_{t-1}) + u_t,\quad
t \in \mathbb{Z},\ d_1 \geq 1,
\]
where $\tilde{S}$ is a measurable and twice continuously differentiable function,
$\bm{\alpha} \in \mathbb{R}^{p_1}$, 
$\tilde{\phi}: \mathbb{R}^{d_1} \to \mathbb{R}^{p_1}$ is a 
measurable function, 
$\bm{\beta} \in \mathbb{R}^{p_2}$, 
$\{\bm{Z}_t\}_{t \in \mathbb{Z}}$ is an 
$\mathbb{R}^{p_2}$-valued covariate process and 
$\{u_t\}_{t \in \mathbb{Z}}$ is a square integrable 
martingale difference sequence with respect to the filtration
$\{\mathcal{F}_t\}$ defined by 
\[
\mathcal{F}_t = \sigma(X_s, \bm{Z}_s; s \leq t),\quad
t \in \mathbb{Z}.
\]
We suppose that 
\[
\E[u_t^2 | \mathcal{F}_{t-1}] = \tilde{\sigma}^2(X_{t-1},\ldots,X_{t-d_2}, \bm{Z}_{t-1}; h),\quad t \in \mathbb{Z},\quad d_2 \geq 1,
\]
where $\tilde{\sigma}$ is a measurable function
and 
$h$ is a possibly infinite-dimensional 
unknown parameter.
Letting $d := d_1 \lor d_2$ and changing the domain of $\tilde{S}$,
$\tilde{\phi}$, and $\tilde{\sigma}$,
we write 
\[
X_t = S(\bm{Y}_{t-1}; \bm{\theta})
+ u_t,\quad
\E[u_t^2 | \mathcal{F}_{t-1}] = \sigma^2(\bm{Y}_{t-1}; h),
\]
where $\bm{X}_{t-1} = (X_{t-1},\ldots,X_{t-d})^\top$, 
$\bm{Y}_{t-1} = (\phi(\bm{X}_{t-1})^\top, \bm{Z}_{t-1}^\top)^\top$, 
\[
S(\bm{Y}_{t-1}; \bm{\theta}) = \tilde{S}(\bm{\alpha}^\top \tilde{\phi}(X_{t-1},\ldots,X_{t-d_1}), \bm{\beta}^\top \bm{Z}_{t-1}),
\]
and
\[ 
\sigma^2(\bm{Y}_{t-1}; h) = \tilde{\sigma}^2(X_{t-1},\ldots,X_{t-d_2}, \bm{Z}_{t-1}; h).
\]
Let $\bm{\theta}_0 = (\bm{\alpha}_0^\top, \bm{\beta}_0^\top)^\top$ be the true value of $\bm{\theta} = (\bm{\alpha}^\top, \bm{\beta}^\top)^\top$,
${\Theta} = {\Theta}_{\alpha} \times {\Theta}_{\beta} \subset \mathbb{R}^{p_1+p_2}$ 
a parameter space for $\bm{\theta}$, and $H$ a
metric space equipped with a metric $d_H$.
Put $T_{10} := \{j : {\alpha}_{0 j} \not=0\}$, 
$T_{20} := \{p_1+j : {\beta}_{0j} \not= 0\}$,
and $T_0 := T_{10} \cup T_{20}$.
Our aim is to estimate $\bm{\theta} = (\bm{\alpha}^\top, \bm{\beta}^\top)^\top$ based on the observation 
$X_{1-d},\ldots,X_n$ and $\bm{Z}_0,\ldots,\bm{Z}_n$.
Put $p= p_1 + p_2$, 
$s=s_{1} + s_{2}$, where $s_{1}$ and $s_{2}$ are the cardinalities of $T_{10}$ and $T_{20}$,
respectively.
Consider the following score functions
\[
\psi_n^{(1)}(\bm{\theta})
= \frac{1}{n} \sum_{t=1}^n 
\frac{\partial}{\partial \bm{\theta}} S(\bm{Y}_{t-1}; \bm{\theta})
\{X_t - S(\bm{Y}_{t-1}; \bm{\theta})\}
\]
and
\[
\Psi_n(\bm{\theta}, h)
=\frac{1}{n} \sum_{t=1}^n \frac{\frac{\partial}{\partial \bm{\theta}} S(\bm{Y}_{t-1};\bm{\theta})}{\sigma^2(\bm{Y}_{t-1}; \hat{h}_n)}
\{X_t - S(\bm{Y}_{t-1};\bm{\theta})\}.
\]
We first construct the estimator $\hat{\bm{\theta}}_n^{(1)}$ for 
$\bm{\theta}$ by the following Dantzig selector type estimator:
\[
\hat{\bm{\theta}}_n^{(1)}
:= \arg \min_{\bm{\theta} \in \mathcal{C}_n} \|\bm{\theta}\|_1,
\quad
\mathcal{C}_{n} := \{\bm{\theta} \in \mathbb{R}^p : 
\|\psi_n^{(1)}(\bm{\theta})\|_\infty \leq \lambda_n\},
\]
where $\lambda_n$ is a tuning parameter.
Moreover, we define the following estimator 
$\hat{T}_n$ for $T_0$:
\[
\hat{T}_n := \{j: |\hat{{\theta}}_{nj}^{(1)}|>\tau_n\},
\]
where $\tau_n$ is a threshold.
Define the Hessian matrices
$\bm{V}_n^{(1)}(\bm{\theta})$ and $\bm{V}_n^{(2)}(\bm{\theta}, h)$
as follows:
\begin{eqnarray*}
\bm{V}_n^{(1)}(\bm{\theta})
&:=& \frac{1}{n} \sum_{t=1}^n \frac{\partial^2}{\partial \bm{\theta} \partial \bm{\theta}^\top} S(\bm{Y}_{t-1}; \bm{\theta})\left\{
X_t - S(\bm{Y}_{t-1}; \bm{\theta})
\right\} \\
&&- \frac{1}{n} \sum_{t=1}^n \frac{\partial}{\partial \bm{\theta}} S(\bm{Y}_{t-1}; \bm{\theta}) \frac{\partial}{\partial \bm{\theta}^\top} S(\bm{Y}_{t-1}; \bm{\theta}),
\end{eqnarray*}
\begin{eqnarray*}
\bm{V}_n^{(2)}(\bm{\theta}, h)
&:=& \frac{1}{n} \sum_{t=1}^n \frac{\partial^2}{\partial \bm{\theta} \partial \bm{\theta}^\top}\frac{S(\bm{Y}_{t-1}; \bm{\theta})}{\sigma^2(\bm{Y}_{t-1}; h)}\left\{
X_t - S(\bm{Y}_{t-1}; \bm{\theta})
\right\} \\
&&- \frac{1}{n} \sum_{t=1}^n  \frac{\frac{\partial}{\partial \bm{\theta}}S(\bm{Y}_{t-1}; \bm{\theta})\frac{\partial}{\partial \bm{\theta}^\top}S(\bm{Y}_{t-1}; \bm{\theta})}{\sigma^2(\bm{Y}_{t-1}; h)}.
\end{eqnarray*}
To establish the asymptotic behavior of the 
estimators $\hat{\bm{\theta}}_n^{(1)}$ and $\hat{T}_n$, we assume the following conditions.
\begin{assumption}\label{regularity time series}
\begin{itemize}
\item[(i)]
There exists a constant $K_{p, s}$
possibly depending on $p$ and $s$ such that 
\[
\E\left[
\left\|
\frac{\partial}{\partial \bm{\theta}} S(\bm{Y}_{t-1}; \bm{\theta})
\right\|_\infty^2
\sigma^2(\bm{Y}_{t-1}; h_0)
\right]\leq K_{p, s}.
\]
\item[(ii)]
It holds that
\[
\log p = o(n),\quad
n \to \infty.
\]
The tuning parameter ${\lambda}_n$ satisfies that
\[
\frac{K_{p, s}}{{\lambda}_n} \sqrt{\frac{\log p}{ n}} \to 0,\quad n \to \infty.
\]
\item[(iii)]
Let $c_n := \|\psi_n^{(1)}(\hat{\bm{\theta}}_n^{(1)})-\psi_n^{(1)}(\bm{\theta}_0)\|_\infty$
and $\theta_{\min} := \inf_{j \in T_0} |\theta_{0j}|$.
Then, it holds that 
\[
\theta_{\min} > 4c_n/\delta.
\]
\end{itemize}
\end{assumption}
The sufficient condition for 
$\|\psi_n^{(1)}(\bm{\theta}_0)\|_\infty$ to obtain the rate of convergence of $\hat{\bm{\theta}}_n^{(1)}$
can be verified as follows.
\begin{prop}\label{gradient time series}
Under Assumption \ref{regularity time series},
it holds that 
\[
\Prob(\|\psi_n^{(1)}(\bm{\theta}_0)\|_\infty > {\lambda}_n) \to 0,\quad
n \to \infty.
\]
\end{prop}
See Subsection \ref{subsec: proof of sec3} for the proof.
Note that we can also use 
the Bernstein inequality for martingales to deduce a non-asymptotic inequality;
see, e.g., \cite{freedman1975tail} and 
for details.
On the other hand,  as the condition for $\bm{V}_n^{(1)}$ strongly depends on the 
model structure such as stationarity or tail property of the process, 
it is difficult to verify it in general settings.
Therefore, we assume the following condition.
\begin{assumption}\label{time series matrix condition}
There exists a constant $\delta>0$ such that 
\begin{equation}\label{matrix condition general time series}
\Prob\left(
\inf_{\bm{0} \not= \bm{v} \in C_{T_0}}\inf_{\bm{\theta} \in {\Theta}}
 \frac{|\bm{v}^\top \bm{V}_n^{(1)}(\bm{\theta}) \bm{v} |}{\|\bm{v} _{T_0}\|_1 \| \bm{v} \|_\infty}>\delta
\right) \to 1,\quad 
n \to \infty.
\end{equation}
\end{assumption}
Note that the left-hand side of \eqref{matrix condition general time series} may depend on 
$n$, $p$, and $s$.
We ensure this condition for
integer-valued autoregressive models with 
large order in Subsection 3.2.

Then, we obtain the following rate of convergence.
\begin{thm}\label{DS time series model thm}
Let Assumptions \ref{regularity time series} 
and \ref{time series matrix condition} hold.
Then, it holds that
\begin{equation}\label{error bound time series}
\Prob\left(\|\hat{\bm{\theta}}_n^{(1)} - \bm{\theta}_0\|_\infty > \frac{2}{\delta}c_n\right)
\to 0,\quad
n \to \infty,
\end{equation}
where $c_n := \|\psi_n^{(1)}(\hat{\bm{\theta}}_n^{(1)})-\psi_n^{(1)}(\bm{\theta}_0)\|_\infty$.
Especially, it holds that 
\[
\|\hat{\bm{\theta}}_n^{(1)} - \bm{\theta}_0\|_\infty = O_p(\lambda_n),\quad n \to \infty.
\]
Moreover, if the threshold 
$\tau_n$ satisfies that 
$2 c_n / \delta< \tau_n < \theta_{\min}/2$, 
it holds that
\begin{equation}\label{selection time series}
\Prob(\hat{T}_n = T_0) \to 1,\quad n \to \infty.
\end{equation}
\end{thm}
For the second step, we construct a consistent 
estimator for $h$ under the following condition:
\begin{assumption}\label{sparsity variance time series}
The function $\sigma^2$ satisfies that 
\[
\sigma^2(\bm{y}; h)
= \sigma^2(\bm{y}_{(T_0)}; h),\quad
\forall \bm{y} \in \mathbb{R}^p,\ h \in H.
\]
\end{assumption}
This condition means that
$\sigma^2(\bm{Y}_{t-1}; h)$ depends only
on $\phi(\bm{X}_{t-1 (T_{10})})$ and $\bm{Z}_{t-1 T_{20}}$ for every $t$.
For example, if $\bm{h}$ is a 
low-dimensional parameter, we can construct an estimator as  
a minimizer of the following function;
\begin{eqnarray}\label{estimate h time series}
\mathcal{B}_n(h)
&=& \frac{1}{n} \sum_{t=1}^n \left|
|X_t - S(\bm{Y}_{t-1}; \hat{\bm{\theta}}_n^{(1)})|^2
- \sigma^2(\bm{Y}_{t-1}; h)
\right|^2.
\end{eqnarray}
For the case where $h$ is a high-dimensional or 
infinite-dimensional parameter, 
we can consider the estimation methods 
as well as Examples \ref{ex2 variance} and \ref{ex3 variance}.
Hereafter, assume that $\hat{h}_n$ is a consistent 
estimator for $h$.

Finally, using 
$\hat{T}_n$ and $\hat{h}_n$, we can construct 
estimators $\tilde{\bm{\theta}}_n^{(j)},\ j=1,2$ 
as solutions to 
the following estimating equations:
\[
\psi_{n \hat{T}_n}^{(1)}(\bm{\theta}_{\hat{T}_n}) = \bm{0},\quad 
\tilde{\bm{\theta}}_{n \hat{T}_n^c} = \bm{0}
\]
and 
\[
\Psi_{n \hat{T}_n}(\bm{\theta}_{\hat{T}_n}, \hat{h}_n) = \bm{0},\quad 
\tilde{\bm{\theta}}_{n \hat{T}_n^c} = \bm{0},
\]
respectively, where for every $T = T_1 \cup T_2$,
\[
\psi_{nT}^{(1)}(\bm{\theta}_T)
= \frac{1}{n} \sum_{t=1}^n \frac{\partial}{\partial \bm{\theta}_T} 
S(\bm{Y}_{t-1(T)};\bm{\theta}_{(T)}),
\]
\[
\Psi_{nT}(\bm{\theta}_T, \hat{h}_n)
 = \frac{1}{n} \sum_{t=1}^n \frac{\frac{\partial}{\partial \bm{\theta}_T} S(\bm{Y}_{t-1(T)};\bm{\theta}_{(T)})}{\sigma^2(\bm{Y}_{t-1 (T)}; \hat{h}_n)}
\{X_t - S(\bm{Y}_{t-1(T)};\bm{\theta}_{(T)})\}.
\]
To establish the asymptotic normality of 
$\tilde{\bm{\theta}}_n^{(k)}, k=1, 2$, we assume the 
following conditions which correspond to 
conditions B1--B6 of \cite{nishiyama2009asymptotic}.
\begin{assumption}\label{time series regularity conditions 2step}
\begin{itemize}
\item[(i)]
There exists a constant $c>0$ such that 
\[
\inf_{\bm{y}_{T_0} \in \mathbb{R}^s, h \in H} \sigma^2(\bm{y}_{(T_0)}; h) > c.
\]
\item[(ii)]
The sparsity $s$ is independent of $n$ and $p$.
The parameter space ${\Theta}_{T_0} \subset \mathbb{R}^s$ is compact and the true value $\bm{\theta}_{0 T_0}$ is an interior point of ${\Theta}_{T_0}$.
There exists a measurable function 
$\Lambda_1$ on $\mathbb{R}^p$ such that 
\begin{eqnarray*}
\lefteqn{
S(\bm{y}_{(T_0)};\bm{\theta}_{(T_0)}) - 
S(\bm{y}_{(T_0)};\bm{\theta}_{0(T_0)})}\\
&=& \frac{\partial}{\partial \bm{\theta}_{T_0}}S(\bm{y}_{(T_0)};\bm{\theta}_{0(T_0)})
(\bm{\theta}_{T_0} - \bm{\theta}_{0 T_0})
+ \Lambda_1(\bm{y}_{(T_0)})\epsilon(\bm{y}_{T_0}; \bm{\theta}_{T_0}, \bm{\theta}_{0 T_0})
\end{eqnarray*}
for every $\bm{x} \in \mathbb{R}^d, \bm{z} \in \mathbb{R}^{p_2}$ and $\bm{\theta} \in {\Theta}$ where $\epsilon$ is a 
measurable function satisfying that 
\[
\sup_{\bm{x} \in \mathbb{R}^d, \bm{z} \in \mathbb{R}^{p_2}}
|\epsilon(\bm{y}_{T_0}; \bm{\theta}_{T_0}, \bm{\theta}_{0 T_0})|
= o(\|\bm{\theta}_{T_0}-\bm{\theta}_{0 T_0}\|_2),\quad
\bm{\theta}_{T_0} \to \bm{\theta}_{0 T_0}.
\]
\item[(ii)]
There exists a measurable function $\Lambda_2$ such that 
\[
\sup_{\bm{\theta} \in \bm{\theta}_{T_0}}|S(\bm{y}_{(T_0)};\bm{\theta}_{(T_0)})|
\leq \Lambda_2(\bm{y}_{T_0});
\]
\[
\sup_{\bm{\theta} \in \bm{\theta}_{T_0}}\left\|\frac{\partial}{\partial \bm{\theta}_{T_0}}S(\bm{y}_{(T_0)};\bm{\theta}_{(T_0)})\right\|
\leq \Lambda_2(\bm{y}_{T_0});
\]
\[
\sigma^2(\bm{y}_{(T_0)}; h_0)
\leq \Lambda_2(\bm{y}_{T_0});
\]
\begin{eqnarray*}
\lefteqn{
\left\|\frac{\partial}{\partial \bm{\theta}_{T_0}}S(\bm{y}_{(T_0)}; \bm{\theta}_{1(T_0)})
- \frac{\partial}{\partial \bm{\theta}_{T_0}}S(\bm{y}_{(T_0)}; \bm{\theta}_{2(T_0)})\right\|_2}\\
&\leq& \Lambda_2(\bm{y}_{T_0})\|\bm{\theta}_{1 T_0} - \bm{\theta}_{2 T_0}\|_2,\quad
\forall \bm{\theta}_{1 T_0}, \bm{\theta}_{2 T_0} \in \bm{\theta}_{T_0};
\end{eqnarray*}
\[
|\sigma^2(\bm{y}_{(T_0)}; h_1)-\sigma^2(\bm{y}_{(T_0)}; h_2)|
\leq \Lambda_2(\bm{y}_{T_0})d_H(h_1, h_2),\quad
\forall h_1, h_2 \in H ;
\]
for every $\bm{x} \in \mathbb{R}^d, \bm{z} \in \mathbb{R}^{s_2}$
\item[(iii)]
The process $\{\bm{Y}_{t T_0}\}_{t \geq 0}$
is ergodic under the true value $(\bm{\theta}_0, h_0)$
with an invariant measure $\mu_{s}$ such that 
for every $\mu_{s}$-integrable function $g$, 
it holds that 
\[
\frac{1}{n}\sum_{t=1}^n g(\bm{Y}_{t T_{0}})
\to^p \int_{\mathbb{R}^{s}}
g(\bm{y}_{T_0}) \mu_{s}(\dr \bm{y}_{T_0}),
\]
as $n \to \infty$.
We also assume that 
\[
\int_{\mathbb{R}^{s}}
|\Lambda_i(\bm{y}_{T_0})|^5 \mu_{s}(\bm{y}_{T_0}) < \infty,
\quad
\int_{\mathbb{R}} |x_0|^4 \mu_1( \dr x_0) < \infty.
\]
\item[(iv)]
The following $s \times s$ matrices 
$\bm{\mathcal{I}}^{(1)}(\bm{\theta}_0, h_0)$, 
$\bm{\mathcal{J}}^{(1)}(\bm{\theta}_0)$ and 
$\bm{\mathcal{I}}^{(2)}(\bm{\theta}_0, h_0)$
are positive definite:
\begin{eqnarray*}
\bm{\mathcal{I}}^{(1)}(\bm{\theta}_0, h_0)
&=& \int_{\mathbb{R}^{s}}
\frac{\partial}{\partial \bm{\theta}_{T_0}}S(\bm{y}_{(T_0)}; \bm{\theta}_{0(T_0)})
\frac{\partial}{\partial \bm{\theta}_{T_0}^\top}S(\bm{y}_{(T_0)}; \bm{\theta}_{0(T_0)})\sigma^2(\bm{y}_{(T_0)}; h_0)
 \mu_{s}( \dr \bm{y}_{T_0}),
\end{eqnarray*}
\[
\bm{\mathcal{J}}^{(1)}(\bm{\theta}_0)
= \int_{\mathbb{R}^{s}}
\frac{\partial}{\partial \bm{\theta}_{T_0}}S(\bm{y}_{(T_0)}; \bm{\theta}_{0(T_0)})
\frac{\partial}{\partial \bm{\theta}_{T_0}^\top} S(\bm{y}_{(T_0)}; \bm{\theta}_{0(T_0)})
 \mu_{s}(\dr \bm{y}_{T_0})
\]
and
\begin{eqnarray*}
\bm{\mathcal{I}}^{(2)}(\bm{\theta}_0, h_0)
&=& \int_{\mathbb{R}^{s}}
\frac{\frac{\partial}{\partial \bm{\theta}_{T_0}}S(\bm{y}_{(T_0)}; \bm{\theta}_{0(T_0)}) \frac{\partial}{\partial \bm{\theta}_{T_0}^\top} S(\bm{y}_{(T_0)}; \bm{\theta}_{0(T_0)})}{\sigma^2(\bm{y}_{(T_0)}; h_0)} \mu_{s}(\dr \bm{y}_{T_0}).
\end{eqnarray*}
\item[(v)]
For every $\epsilon>0$,
\begin{eqnarray*}
\inf_{\bm{\theta}_{T_0}: \|\bm{\theta}_{T_0}-\bm{\theta}_{0 T_0}\|_2>\epsilon}
\left\|
\int_{\mathbb{R}^{s}}
\frac{\frac{\partial}{\partial \bm{\theta}_{T_0}}S(\bm{y}_{(T_0)}; \bm{\theta}_{0(T_0)})}{\sigma^2(\bm{y}_{(T_0)}; h_0)}\left[
S(\bm{y}_{(T_0)}; \bm{\theta}_{(T_0)})
-S(\bm{y}_{(T_0)}; \bm{\theta}_{0(T_0)})
\right]\mu_{s}(\dr \bm{y}_{T_0})
\right\|_2
> 0.
\end{eqnarray*}
\item[(vi)]
The metric entropy condition for $(H, d_H)$ is 
satisfied:
\[
\int_0^1 \sqrt{\log N(H, d_H, \epsilon)} \dr \epsilon 
< \infty.
\]
\item[(vii)]
The estimator $\hat{h}_n$ is consistent estimator of 
$h_0$, i.e., 
$d_H(\hat{h}_n, h_0) = o_p(1)$ as $n \to \infty$.
\end{itemize}
\end{assumption}
Some discussions about the 
metric entropy condition (vi) and examples 
satisfying such conditions are provided in, e.g.,  
\cite{nishiyama2009asymptotic,vanwellner1996}.
Sufficient conditions for the consistency of 
$\hat{h}_n$ is described in 
Theorem 5.4 of \cite{nishiyama2009asymptotic} when 
$\hat{h}_n$ is given as a solution to $\mathcal{B}_n(h)=0$, where $\mathcal{B}_n(h)$ is defined in 
\eqref{estimate h time series}.
Under above assumptions, we can apply Theorem \ref{normality general model} to establish the asymptotic normality of 
$\tilde{\bm{\theta}}_n^{(k)}, k=1, 2$ as follows.
\begin{lem}\label{normality time series}
Suppose that Assumptions \ref{regularity time series} ,\ref{time series matrix condition}, 
\ref{sparsity variance time series},
and \ref{time series regularity conditions 2step} hold.
Assume moreover that
$2c_n/\delta <\tau_n < \theta_{\min}/2$.
Then, it holds that  
\[
\sqrt{n} \left(\tilde{\bm{\theta}}_{n \hat{T}_n}^{(k)} - \bm{\theta}_{0 T_0}\right) 1_{\left\{\hat{T}_n = T_0\right\}}
\to^d \mathcal{N}_s(0, \bm{\Sigma}^{(k)}),\quad
n \to \infty
\]
for $k=1, 2$,
where 
\[
\bm{\Sigma}^{(1)}
= \bm{\mathcal{J}}^{(1)}(\bm{\theta}_0)^{-1} \bm{\mathcal{I}}^{(1)}(\bm{\theta}_0, h_0) \bm{\mathcal{J}}^{(1)}(\bm{\theta}_0)^{-1}
\]
and 
\[
\bm{\Sigma}^{(2)}
= \bm{\mathcal{I}}^{(2)}(\bm{\theta}_0, h_0)^{-1}.
\]
\end{lem}
\begin{proof}
It follows from Lemmas 5.1, 5.2, and Theorem 5.3 
of \cite{nishiyama2009asymptotic} that 
\[
\sqrt{n} \left(\tilde{\bm{\theta}}_{n {T_0}}^{(k)} - \bm{\theta}_{0 T_0}\right) \to^d \mathcal{N}_s(0, \bm{\Sigma}^{(k)}),\quad
n \to \infty.
\]
Combining this fact with \eqref{selection time series},  we obtain the conclusion as well as 
Theorem \ref{normality general model}.
\end{proof}
\begin{rem}
We can show that 
$
\bm{\Sigma}^{(1)} - \bm{\Sigma}^{(2)}
$
is non-negative definite,
which implies $\tilde{\bm{\theta}}_n^{(2)}$ is asymptotically
more efficient than $\tilde{\bm{\theta}}_n^{(1)}$.
\end{rem}
The weak convergence of the proposed
estimator in  the Hilbert space $\ell^2$ 
can be shown as well as Section \ref{sec:general}.
Consider the $\ell^2$-valued random sequence 
$\{\bm{\mathcal{R}}_n^{(k)}\}_{n \in \mathbb{N}}$
for $k=1, 2$ given by 
\[
\langle \bm{e}_j, \bm{\mathcal{R}}_n^{(k)}\rangle
= \begin{cases}
\sqrt{n} (\tilde{\bm{\theta}}_{nj}^{(k)} - \bm{\theta}_{0j} ) & 
1 \leq j \leq p \\
0 & j > p
\end{cases}.
\] 
Recall that $\{\bm{e}_j\}_{j \in \mathbb{N}}$ is the 
canonical basis of $\mathbb{R}^\infty$.
\begin{lem}\label{cramer-wold time series}
Under the same assumptions as Lemma \ref{normality time series}, 
it holds that 
\[
\langle \bm{u}, \bm{\mathcal{R}_n}^{(k)} \rangle
\to^d \mathcal{N}(0, \bm{u}_{T_0}^\top 
\bm{\Sigma}^{(k)}\bm{u}_{T_0}),\quad
n \to \infty,
\]
for every $\bm{u} \in \ell^2$ and 
$k=1, 2$.
\end{lem}
\begin{thm}\label{weak convergence l2 time series}
Under the same assumptions as Lemma \ref{normality time series}, it holds that 
\[
\bm{\mathcal{R}}_n^{(k)} \to^d \bm{\mathcal{R}}^{(k)},\quad n \to \infty \quad
\mbox{in}\ \ell^2,
\]
where 
$\bm{\mathcal{R}}^{(k)}$ is a centered Gaussian random field whose 
distribution is determined by 
the distribution of its marginals 
\[
\langle \bm{u}, \bm{\mathcal{R}}^{(k)} \rangle \sim
\mathcal{N}(0, \bm{u}_{T_0}^\top\bm{\Sigma}^{(k)}\bm{u}_{T_0})
\]
for every $\bm{u} \in \ell^2$ and $k=1, 2$.
\end{thm}
Since Lemma \ref{cramer-wold time series}
and Theorem \ref{weak convergence l2 time series} can be proved as well as 
Lemma \ref{cramer-wold general}
and Theorem \ref{weak convergence l2 general}, respectively,
we omit their proofs.
\subsection{Example: Univariate integer-valued autoregressive model of large order}\label{subsec:INAR}
As an example, let us consider the following model:
\begin{equation}\label{INAR model}
X_t = \sum_{i=1}^{p} {\alpha}_i \circ X_{t-i} + 
\epsilon_t,\quad
t \in \mathbb{Z},
\end{equation}
where ${\alpha}_i \geq 0, i=1,\ldots,p$, and, for a non-negative 
integer-valued random variable $X_{t-1}$,
\[
{\alpha}_i \circ X_{t-1} = \begin{cases}
\sum_{j=1}^{X_{t-1}} \xi_j^{(i,t)} & X_{t-1}>0\\
0 & X_{t-1}=0,
\end{cases}
\]
with i.i.d.~non-negative integer-valued random sequences 
$\{\xi_j^{(i,t)}\}_{j \in \mathbb{N}}\ i=1,\ldots,p,\ t \in \mathbb{Z}$ 
which is independent over $i$ and $t$
with 
$\E[\xi_j^{(i,t)}]={\alpha}_i,\ \Var[\xi_j^{(i,t)}]={\beta}_i$
for every $i =1,\ldots,p$ and $t \in \mathbb{Z}$,
and 
$\{\epsilon_t\}_{t \in \mathbb{Z}}$ is an 
i.i.d.~non-negative integer-valued random sequence which is independent of $X_s, s < t$
with $\E[\epsilon_t] = \mu_\epsilon$ and $\Var[\epsilon_t] = \sigma_\epsilon^2$.
We call $\{\xi_j^{(i,t)}\}_{j \in \mathbb{N}},  t \in \mathbb{Z}$ counting sequences correspond to $\alpha_i \circ$ for every 
$i = 1,\ldots,p$.
Suppose that for every $i \not=i'$, 
$\{\xi_j^{(i,t)}\}_{j \in \mathbb{N}}$
and $\{\xi_j^{(i',t)}\}_{j \in \mathbb{N}}$
are independent.
Let $\{\mathcal{F}_{t}\}_{t \in \mathbb{Z}}$
be a filtration defined by 
\[
\mathcal{F}_{t} = \sigma(X_s; s \leq t),\quad
t \in \mathbb{Z},
\]
$\bm{\theta} = (\mu_\epsilon, {\alpha}_1,\ldots, {\alpha}_p)^\top$ and $\bm{h} = (\sigma^2_\epsilon, {\beta}_1,\ldots,{\beta}_p)^\top$.
We define an $\{\mathcal{F}_t\}_{t \in \mathbb{Z}}$-martingale $\{u_t\}_{t \in \mathbb{Z}}$ as follows:
\[
u_t := X_t - \E[X_t | \mathcal{F}_{t-1}],\quad t \in \mathbb{Z}.
\]
Noting that 
\[
\E[X_t | \mathcal{F}_{t-1}] = \mu_\epsilon + \sum_{i=1}^p {\alpha}_i X_{t-i},
\]
we have the following representation:
\[
X_t = \mu_{\epsilon} + \sum_{i=1}^p \alpha_i X_{t-i} + u_t,
\]
where 
$\{u_t\}_{t \in \mathbb{Z}}$ satisfies
\[
\E[u_t^2 | \mathcal{F}_{t-1}]
= \sigma^2_\epsilon + \sum_{i=1}^p {\beta}_i X_{t-1}
=: \sigma^2(X_{t-1},\ldots,X_{t-p}; \bm{h}).
\]
Let $\bm{\theta}_0$ and $\bm{h}_0$ be the true values of 
$\bm{\theta}$ and $\bm{h}$, respectively. 
Let us assume that 
${\theta}_{0i} \not= 0$ when $i \leq s$ and 
${\theta}_{0i} =0$ when $i >s$ for some $s>1$.
Since $\{\xi_j^{(i,t)}\}$ and $\{\epsilon_t\}$ are non-negative integer-valued random sequences, 
$\E[\xi_j^{(i,t)}] = {\alpha}_i =0$ or $\E[\epsilon_t] = \mu_\epsilon =0$ implies that these random variables are degenerate, i.e., 
$\Var[\xi_j^{(i,t)}]={\beta}_i = 0$ or $\Var[\epsilon_t]=\sigma^2_\epsilon=0$.
Therefore, we can see 
$T_0 = \{j : {\theta}_{0j} \not=0\}\supset \{j : h_{0 j} \not=0\}$, which means the sparsity of $\bm{h}_0$ is 
at most $s.$

Our aim is to consider the order selection and to construct an asymptotically good estimator of 
the $\bm{\theta}$ through the Dantzig selector based on the observations $X_t, t=1-p,\ldots,n$.
The score functions and the Hessian matrices are given as follows:
\[
\psi_n^{(1)}(\bm{\theta})
= \frac{1}{n}\sum_{t=1}^n \bm{Y}_{t-1} \left\{
X_t - \bm{\theta}^\top \bm{Y}_{t-1}
\right\},\quad
\bm{V}_n^{(1)}(\bm{\theta})
= - \frac{1}{n}\sum_{t=1}^n \bm{Y}_{t-1} \bm{Y}_{t-1}^\top,
\]
\[
\Psi_n(\bm{\theta}, \bm{h})
= \frac{1}{n}\sum_{t=1}^n \frac{\bm{Y}_{t-1}}{\bm{h}^\top \bm{Y}_{t-1}} \left\{
X_t - \bm{\theta}^\top \bm{Y}_{t-1}
\right\},\quad \mbox{and}\quad
\bm{V}_n(\bm{\theta}, h)
= - \frac{1}{n}\sum_{t=1}^n \frac{\bm{Y}_{t-1} \bm{Y}_{t-1}^\top}{\bm{h}^\top \bm{Y}_{t-1}},
\]
where 
$\bm{Y}_{t-1}=(1,X_{t-1},\ldots,X_{t-p})^{\top}$.
Let $\bm{V} = \E[\bm{Y}_{0} \bm{Y}_{0}^\top]$.
We assume the following conditions for $\{X_t\}_{t \in \mathbb{Z}}$.
\begin{assumption}\label{stationarity INAR}
\begin{itemize}
\item[(i)]
It holds that 
\[
\sum_{i=1}^p \bm{\alpha}_{0 i} < 1.
\]
\item[(ii)]
For every $t \in \mathbb{Z}$, 
$X_t$ is $\mathcal{F}_{t-1}$ conditionally 
{\rm sub-Poissonian}, that is, 
\[
\E[\exp(uX_t) | \mathcal{F}_{t-1}]
\leq \exp(\mu_{t-1}( \mathrm{e}^u - 1)),\quad
u >0,
\]
where
\[
\mu_{t-1} = \E[X_t | \mathcal{F}_{t-1}] = \bm{\theta}_0^\top \bm{Y}_{t-1}.
\]
\end{itemize}
\end{assumption}
The condition (i) implies that there exists a strictly stationary and ergodic solution $\{X_t\}_{t \in \mathbb{Z}}$ to \eqref{INAR model}.
The condition (ii) is easily verified 
when, e.g., $\{\xi_j^{(i,t)}\}$ for all $i=1,\ldots,p$ and
$\{\epsilon_t\}$ are Poisson or Bernoulli sequences.
See, e.g., \cite{ahle2022sharp} for tail bounds of sub-Poissonian random variables,  
\cite{du1991integer,neumann2011absolute,
doukhan2012weak,
doukhan2021mixing}, for details of the 
fundamental properties of integer-valued time series.

Now, we introduce the $\tau$-mixing coefficient 
to apply the concentration inequality established by \cite{merlevede2011bernstein}.
For any $\mathbb{R}^d$-valued random variable $X$ on a probability space $(\Omega, \mathcal{F}, \Prob)$ with $\|X\|_{L^1} < \infty$, and any $\sigma$-field $\mathcal{M} \subset \mathcal{F}$, let 
$\Prob_{X | \mathcal{M}}$ be a conditional distribution of $X$ given $\mathcal{M}$ and 
$\Prob_X$ be the distribution of $X$.
Then, we consider the following coefficient:
\[
\tau(\mathcal{M}, X) = 
\left\|\sup_{f \in \Lambda_1(\mathbb{R}^d)}
\int f(\bm{x}) \Prob_{X | \mathcal{M}} (\dr \bm{x})
- \int f(\bm{x}) \Prob_{X} (\dr \bm{x})
\right\|_{L^1},
\]
where 
\[
\Lambda_1(\mathbb{R}^d)
= \left\{f: \mathbb{R}^d \to \mathbb{R} \left| 
\sup_{x \not=y} \frac{|f(\bm{x})-f(\bm{y})|}{\|\bm{x} - \bm{y}\|_1}\right. \leq 1
\right\}.
\]
The $\tau$-mixing coefficients $\tau(i)$, $i \in \N$ of 
a sequence $\{X_t\}$ are then given by
\[
\tau(i)
= \sup_{k \geq 0} \max_{1 \leq l \leq k}
\sup \left\{
\tau(\mathcal{M}_j, (X_{j_1},\ldots,X_{j_l})) |
j+i \leq j_1 \leq \ldots \leq j_l
\right\},
\] 
where $\mathcal{M}_j = \sigma(X_t, t \leq j)$.
We call $\{X_t\}$ is $\tau$-weakly dependent if 
$\tau(l) \to 0$ as $l \to \infty$. 
See, e.g., \cite{doukhan2012weak} for the detail of the $\tau$-coefficients.
Especially, noting that the true order of the INAR 
$\{X_t\}$ is $s$, we have the following lemma.
\begin{lem}\label{beta-mixing INAR}
Suppose that Assumption \ref{stationarity INAR}
holds.
\begin{itemize}
\item[(i)] 
There exist constants $a, c>0$ and $\gamma_1 >0$ such that 
\[
\tau(l) \leq a\exp(-cl^{\gamma_1}),\quad
l \in \mathbb{N}.
\]
Moreover, it holds that 
\[
\gamma := \left(\frac{1}{\gamma_1} + 2 \right)^{-1} <1. 
\]
\item[(ii)]
It holds that 
\[
\beta(l) \leq as\exp(-cl^{\gamma_1}),\quad
l \in \mathbb{N}.
\]
Here, ${\beta}(l), l \in \mathbb{N}$ is the ${\beta}$-mixing coefficient, i.e., 
\[
{\beta}(l) := \sup \frac{1}{2}\sum_{i=1}^I \sum_{j=1}^J
|\Prob(A_i \cap B_j) - \Prob(A_i)\Prob(B_j)|,\quad l \in \mathbb{N},
\]
where the supremum is taken over all pair of partitions
$\{A_i\}_{1 \leq i \leq I} \subset 
\sigma(X_u, u \leq t)$
and $\{B_j\}_{1 \leq j \leq J} \subset \sigma(X_u, u \geq t+l)$ for every $I, J \in \mathbb{N}$.
\end{itemize}
\end{lem}
See Theorems 1 and 2 of \cite{doukhan2012weak} for the proof.

We consider the following conditions.
\begin{assumption}\label{regularity INAR}
\begin{itemize}
\item[(i)]
There exists a positive constant $f_\infty$ 
such that 
\[
F_\infty(T_0; \bm{V})>f_\infty.
\]
\item[(ii)]
The matrices 
\[
\bm{V}_{T_0, T_0} 
= \E[\bm{Y}_{0 T_0}\bm{Y}_{0 T_0}^\top],\quad
\mathcal{I}^{(2)}(\bm{\theta}_0, \bm{h}_0)
= \E\left[\frac{\bm{Y}_{0 T_0} \bm{Y}_{0 T_0}}{\bm{h}_{0T_0}^\top \bm{Y}_{0 T_0}}\right]
\]
are invertible.
\item[(iii)]
The order $p$ of integer-valued autoregressive model and the sparsity $s$ of the true value $\bm{\theta}_0$, and tuning parameter $\lambda_n$ satisfy that 
\[
s= O(n^{\zeta}),\quad
ps\log p = o(n),\quad
\frac{1}{\lambda_n}\sqrt{\frac{ps \log p}{n}}=o(1),\quad
n \to \infty,
\]
where $\zeta \in [0, 2/11\land \gamma/(1+\gamma))$ is a constant.
\end{itemize}
\end{assumption}

Let $K_X$ be a constant defined by 
\[
K_X = \frac{2 \mu_{0\epsilon} + 1}{2(1-\sum_{i=1}^p {\alpha}_{0i})}.
\]

Then, we have the following proposition.

\begin{prop}\label{INAR maximal ineq}
Let Assumptions \ref{stationarity INAR} and  \ref{regularity INAR} hold. 
\begin{itemize}
\item[(i)]
For every $q \geq 1$, it holds that
\[
\|X_t\|_{L^q} \leq q K_X,\quad
t \in \mathbb{Z}.
\]
Moreover, it holds that
\[
\Prob(\|\psi_n^{(1)}(\bm{\theta}_0)\|_\infty>\lambda_n)
\leq \frac{12\sqrt{6} \|\bm{h}_0\|_\infty^{1/2} {K}_X^{3/2}}{\lambda_n}\sqrt{\frac{ps\log(1+p)}{n}}.
\]
\item[(ii)]
It holds that 
\[
\sup_{\bm{v} \in C_{T_0} \cap \mathbb{S}_p} \sup_{q \geq 1}\frac{\E[|\bm{v}^\top \bm{Y}_{t-1}|^q]^{1/q}}{q} \leq K_s,\quad
t \in \mathbb{Z},
\]
where 
\[
K_s = 2 \sqrt{s} (1+K_X)
\]
and $\mathbb{S}_p \subset \mathbb{R}^p$ is a
unit sphere centered at $\bm{0}$. 
\item[(iii)]
For every $0 < \eta < f_\infty/2$, it holds that
\[
\Prob\left(
F_\infty(T_0, \bm{V}_n^{(1)}) > f_\infty - 2 \eta
\right) \to 1,\quad
n \to \infty,
\]
where 
\[
F_\infty(T_0, \bm{V}_n^{(1)})
=\inf_{\bm{v} \in C_{T_0}\setminus\{ \bm{0} \}} \frac{|\bm{v}^\top \bm{V}_n^{(1)} \bm{v} |}{\| \bm{v}_{T_0}\|_1 \|\bm{v} \|_\infty}.
\]
\end{itemize}
\end{prop}
See Subsection \ref{subsec: proof of sec3} for the 
proof.
\begin{rem}
\begin{itemize}
\item[(i)]
In 
\cite{drost2008note}, it is proved that 
if $\{\xi_j^{(i,t)}\}_{j \in \mathbb{N}},\ i=1,\ldots,p$ and 
$\{\epsilon_t\}_{t \in \mathbb{Z}}$ satisfy
$\E[|\xi_j^{(i,t)}|^q]<\infty$ and $\E[|\epsilon_t|^q]<\infty$, it holds that 
$\E[|X_t|^q]<\infty$
under Assumption \ref{regularity INAR}-(i).
On the other hand, we proved the sub-exponential property 
of INAR under the sub-Poissonian condition in Proposition \ref{INAR maximal ineq}-(ii), which may be a new result.
\item[(ii)]
We may consider heavier tailed time series such that 
Assumption \ref{regularity INAR}-(iv) is not satisfied.
For such cases, we can use a self-weighted score function such as
\[
\psi^{(1)}_{nw}(\bm{\theta})
= \frac{1}{n}\sum_{t=1}^n w_{t-1} \bm{Y}_{t-1} \left\{
X_t - \bm{\theta}^\top \bm{Y}_{t-1}
\right\},
\]
where 
$w_t$'s are non-negative $\mathcal{F}_{t}$-measurable random variables for every $t \in \mathbb{Z}$ satisfying
\[
\sup_{\bm{v} \in C_{T_0}\cap \mathbb{S}_p} \sup_{q \geq 1}\frac{\E[|\bm{v}^\top w_{t-1} \bm{Y}_{t-1}|^q]^{1/q}}{q} \leq K_s,\quad
t \in \mathbb{Z}.
\]
for some $K_s > 0$.
\end{itemize}
\end{rem}
As for the nuisance parameter $\bm{h}$, 
we can construct an estimator $\hat{\bm{h}}_n$ by 
a solution to the following estimating equation:
\[
\varphi_{n \hat{T}_n}(\bm{h}_{\hat{T}_n})
= \frac{1}{n}\sum_{t=1}^n \bm{Y}_{t-1 \hat{T}_n}\left\{
|X_t-\hat{\bm{\theta}}_{n \hat{T}_n}^{(1) \top}\bm{Y}_{t-1 \hat{T}_n}|^2
- \bm{h}_{\hat{T}_n}^\top \bm{Y}_{t-1 \hat{T}_n}
\right\} =\bm{0},\quad
\bm{h}_{\hat{T}_n^c} = \bm{0},
\]
where 
\[
\hat{T}_n = \{j : |\hat{{\theta}}_{n j}^{(1)}|> \tau_n\},\quad
\tau_n \in \left(\frac{4 \lambda_n}{\delta}, \frac{\inf_{j \in T_0}|{\theta}_{0j}|}{2}\right).
\] 
Here, we consider the restriction by $\hat{T}_n$
since $\{j : h_{0j} \not=0\}\subset T_0$.
Note that this equation can be solved as follows
\[
\hat{\bm{h}}_{n \hat{T}_n}
= \left(\frac{1}{n}\sum_{t=1}^n \bm{Y}_{t-1\hat{T}_n} 
\bm{Y}_{t-1 \hat{T}_n}^\top\right)^{-1}
\frac{1}{n}\sum_{t=1}^n 
(X_t - \hat{\bm{\theta}}_{n\hat{T}_n}^{(1) \top}\bm{Y}_{t-1 \hat{T}_n})^2\bm{Y}_{t-1 \hat{T}_n},\quad
\hat{\bm{h}}_{n \hat{T}_n^c} = \bm{0},
\]
which is consistent when the Dantzig selector 
type estimator $\hat{\bm{\theta}}_n^{(1)}$ and 
$\hat{T}_n$ are consistent.
Therefore, the two step estimator $\tilde{\bm{\theta}}_n$ is defined as a 
solution to the following equations:
\[
\Psi_{n \hat{T}_n}\left(\bm{\theta}_{\hat{T}_n}, \hat{\bm{h}}_{n \hat{T}_n}\right)
= \frac{1}{n}\sum_{t=1}^n \frac{\bm{Y}_{t-1 \hat{T}_n}}{\hat{\bm{h}}_{n \hat{T}_n}^\top \bm{Y}_{t-1 \hat{T}_n}} \left\{
X_t - \bm{\theta}_{\hat{T}_n}^\top \bm{Y}_{t-1 \hat{T}_n}
\right\}=\bm{0},\ 
\bm{\theta}_{\hat{T}_n^c} = \bm{0}.
\]
In conclusion, under Assumption \ref{regularity INAR}, we can apply the estimation procedure 
described in the previous subsection and obtain an
asymptotically normal estimator when $s$ is independent of $n$ and $p$.

\begin{rem}
\begin{itemize}
\item[(i)]
When $\{\xi_j^{(i,t)}\},\ i=1,\ldots,p$ and 
$\{\epsilon_t\}$ are Poisson sequences, 
we have $\bm{h} = \bm{\theta}$.
Therefore, we can plug $\hat{\bm{\theta}}_n^{(1)}$ itself
in $\Psi_n$ when $\hat{\bm{\theta}}_n^{(1)}$ is consistent.
\item[(ii)]
Here, we consider the INAR $(p)$ model 
with large $p$ as an example.
However, 
we can also apply the general theory to  
INAR models with high-dimensional covariates as well as \cite{wang2020variable}, 
or row-wise estimation for regression 
matrices of 
multivariate INAR models
under appropriate conditions.
We provide some results of the numerical 
simulations for the 
multivariate INAR $(1)$
in Subsection \ref{subsec:simulation}
\end{itemize}
\end{rem}
\subsection{Application to the Hawkes processes}\label{subsec:hawkes}
The Poisson INAR $(p)$ process is known to be a
discrete approximation of the Hawkes process.
More precisely, let us consider the Hawkes process 
$\{N_t\}_{t \in \mathbb{R}}$ with the intensity
\[
\lambda(t) = \eta + \int_{\mathbb{R}} a(t-s) \dr N_s,
\]
where $\eta>0$ and $h: \mathbb{R} \to \mathbb{R}_{+}$ is a non-negative and continuous function with 
$
a(t)=0$ for 
$t \leq 0$.
We assume that the support of $a$ equals to $(0, \tau]$ for some constant $\tau>0$ and that 
\[
\int_{\mathbb{R}} a(t) \dr t < 1.
\]
According to \cite{kirchner2016hawkes}, 
there exist constants $\delta>0$ and $\tilde{K}<1$ such that for any $\Delta \in (0, \delta)$, 
\[
K^{(\Delta)} = \Delta \sum_{k=1}^\infty a(k \Delta)
= \Delta \sum_{k=1}^p a(k \Delta) \leq \tilde{K} <1,
\]
where $p \geq \lceil \tau/\Delta \rceil$.
Let $\{X_t^{(\Delta)}\}$ be an INAR($p$) process 
given by 
\[
X_t^{(\Delta)}
= \sum_{i=1}^p {\alpha}_i^{(\Delta)} \circ X_{t-i}^{(\Delta)} + \epsilon_t^{(\Delta)},
\]
where 
$\{\epsilon_t^{(\Delta)}\}$ is an i.i.d.~$Poisson(\Delta \eta)$ sequence and 
${\alpha}_i^{(\Delta)} = \Delta a(i \Delta)$.
We define the counting process 
$N^{(\Delta)}$ as follows:
\[
N^{(\Delta)}(A)
= \sum_{k : k \Delta \in A} X_k^{(\Delta)},\quad
A \in \mathcal{B}_b(\mathbb{R}),\quad
\Delta \in (0, \delta),
\]
where $\mathcal{B}_b(\mathbb{R})$ is a family
of bounded Borel sets in $\mathbb{R}$.
Then, \cite{kirchner2016hawkes} shows that 
\[
N^{(\Delta)} \to^d N,\quad
\Delta \to 0.
\]
Consider that we observe a Hawkes process 
$\{N_t\}_{t \in [0,T]}$.
Fix a small constant $\Delta = \lceil T/n \rceil$
for $n \in \mathbb{N}$.
We can consider the $n$-partitions of $(0, T]$:
\[
(0,T]
= \bigcup_{k=0}^{n-1} (k\Delta, (k+1)\Delta].
\]
Then, we can see that if $\Delta$ is sufficiently small, 
\[
N((k \Delta, (k+1)\Delta])
\approx X_k^{(\Delta)},
\]
where 
$\{X_k^{(\Delta)}\}$ is a INAR($p$) process 
with sufficiently large $p$.
Conditional
least squares estimator based on this approximation
is introduced in \cite{kirchner2017estimation}.
On the other hand, we can also apply the estimation method discussed in Subsection \ref{subsec:INAR}.
Let $\tau$ be a constant such that 
$\tau \leq p \Delta \leq T$ and 
$s = \lceil \tau/\Delta \rceil$.
Then, $\bm{\alpha}^{(\Delta)}=({\alpha}_1^{(\Delta)},\ldots,{\alpha}_p^{(\Delta)})^\top$ satisfies that
\[
\bm{\alpha}^{(\Delta)} = ({\alpha}_1^{(\Delta)},\ldots,{\alpha}_s^{(\Delta)},0,\ldots,0)^\top \in \mathbb{R}^p.
\]
Therefore, we can estimate $\bm{\alpha}^{(\Delta)}$ and 
the support $\tau = s \Delta$ by using the Dantzig selector and proposed estimator.

\subsection{Simulation studies}\label{subsec:simulation}
We consider the 
following model:
\begin{equation}\label{sim model uni}
X_t|\mathcal{F}_{t-1}
\sim Poisson(\lambda_t),
\quad
\lambda_t = \mu_\epsilon + \sum_{i=1}^p \alpha_i X_{t-i}
\end{equation}
and $\{\mathcal{F}_t\}_{t \in \mathbb{Z}}$ is a filtration defined by 
$\mathcal{F}_{t} = \sigma(X_s; s \leq t)$.
Note that this model can be considered as a special 
case of univariate integer-valued time series of 
order $p$;
\[
X_t = \sum_{i=1}^p \alpha_i \circ X_{t-i} + \epsilon_t,\quad t \in \mathbb{Z},
\]
where $\alpha_i \circ, i=1,\ldots$ are 
Poisson reproduction operators and 
$\{\epsilon\}_{t \in \mathbb{Z}}$
is an i.i.d. integer-valued 
innovation process.
If $\epsilon_t$ and the counting sequences correspond to $\alpha_i \circ, i=1,\ldots,p$ are independent 
Poisson sequences, the conditional distribution of the process $\{X_t\}_{t \in \mathbb{Z}}$ should be Poisson.
The true value is given by 
\[
\bm{\alpha}_0 = (0.3,0.2,0.2,0.2,0,\ldots,0)^\top \in \mathbb{R}^p,\quad
\mu_\epsilon = 0.5.
\]
The scenario of the simulation study for 
the model \eqref{sim model uni} is given as follows.
\begin{description}
\item[Case 1.]
The order $p$ of the INAR model is $10$, and the number $n$ of 
the observation points is $1000$ or $2000$.
\item[Case 2.]
$p=20$, $n=1000$ or $n=2000$.
\end{description}
Moreover, let us consider the following 
$p$-dimensional time series:
\begin{equation}\label{sim model multi}
\bm{Y}_{t, j} | \mathcal{F}_{t-1}
\sim Poisson (\bm{\lambda}_{t, j}),\quad
\bm{\lambda}_t = \bm{\eta} + \bm{A} \bm{Y}_{t-1}.
\end{equation}
for every $j=1,\ldots,p$ 
and $\bm{A} \in \mathbb{R}^{p\times p}$ is a 
coefficient matrix whose elements are all nonnegative.
Let $X_t = \bm{Y}_{t, 1}$. 
Then, we have 
\[
X_t | \mathcal{F}_{t-1} \sim
Poisson({\bm{\lambda}_{t, 1}}),\quad
\bm{\lambda}_{t, 1} = \bm{\eta}_1 + \bm{\alpha}^\top \bm{Y}_{t-1},
\]
where $\bm{\alpha}$ is the first row of $\bm{A}$.
Note that this model can be considered as a special 
case of multivariate integer-valued time series of 
order 1;
\[
\bm{Y}_t = \bm{A} \circ \bm{Y}_{t-1} + \bm{\epsilon}_t,
\]
where 
$\bm{A} \circ$ is a multivariate 
reproduction operator
defined in \cite{latour1997multivariate} and 
$\{\bm{\epsilon}_t\}_{t \in \mathbb{Z}}$
is an i.i.d. innovation process.
If $\bm{\epsilon}_t$ and the counting sequences 
correspond to $\bm{A} \circ $ are independent 
Poisson sequences, 
the conditional distributions of 
each component should be also Poisson.

The true values are given by 
\[
\bm{A} = \diag(\bm{A}_1,\ldots,\bm{A}_q),\quad
\bm{A}_i = \left(\begin{array}{cccc}
0.3& 0.2& 0.2& 0.2  \\
0.2& 0.3& 0.2& 0.2  \\
0 & 0.2 & 0.3& 0.2 \\
0 & 0 & 0.2 & 0.3  \\
\end{array}
\right),\quad
i=1,\ldots,q,
\]
and
\[
\bm{\eta} = (0.5,\ldots,0.5)^\top.
\]
We consider the following settings:
\begin{description}
\item[Case 3.]
The dimension $p$ is $100$, 
and the number $n$ of the observation points is $1000$ or $2000$.
\item[Case 4.]
$p=200$, $n = 1000$ or $n=2000$.
\end{description}
Our aim is to estimate $\bm{\alpha}$
for each case.
We choose the tuning parameter $\lambda_n$
by a cross-validation method while we put the 
threshold $\tau_n=0.05$, 
which is determined by taking into account of 
Remark \ref{choice of threshold}.
To construct the Dantzig selector for 
$\bm{\alpha}$,
we first centralize the process, then, 
we apply the $\mathtt{gds}$ method in the package $\mathtt{hdme}$ in $\mathtt{R}$; see \cite{sorensen2019hdme}.
The two step estimator $\tilde{\bm{\alpha}}_n$ is constructed by 
weighted least squares method described 
in the previous subsection, 
where the intercept term $\eta$ (Cases 1 and 2) and $\eta_1$ (Cases 3 and 4) 
are estimated by least squares method 
after variable selection.

Tables \ref{sim c1 uni}-\ref{sim c2 multi} 
illustrate the 
average of $l_\infty$ and $l_2$ errors of the 
estimators, empirical probability of the success of the selection and the $p$-values of 
the Royston test in the package $\mathtt{MVN}$ in $\mathtt{R}$ (see \cite{korkmaz2014mvn, mecklin2005monte}), for the 
multivariate normality of the two step 
estimator restricted by the true support index set, over 500 replications.
For all cases, the Dantzig selector 
works well in terms of errors and selection.
Moreover, the two step estimator $\tilde{\bm{\alpha}}_{n}$ has smaller errors than $\hat{\bm{\alpha}}_n^{(1)}$
and seems to satisfy the normality when $n$ is large.

We further check the histograms of 
$\sqrt{n} \bm{u}^\top (\tilde{\bm{\alpha}}_n- \bm{\alpha}_0)$, 
$\bm{u} = \tilde{\bm{u}}/\|\tilde{\bm{u}}\|_2$, where each component of 
$\tilde{\bm{u}}$ is generated from 
the uniform distribution between $[-1,1]$ 
for each case.
Figure \ref{fig:hist} shows the result.
The normality seems to be satisfied, 
but some outliers are appeared.
We can confirm that the outlier occurred when the variable selection failed.
Therefore, the numerical results seem consistent with Theorem \ref{cramer-wold time series}.

\newpage
\begin{table}[!h]
\begin{center}
\small
\caption{Estimations for regression coefficients over $500$ replications based on the model \eqref{sim model uni}: Case 1.}
\begin{tabular}{c|cc}
 & $(n, p)=(1000,10)$ & $(n, p) = (2000,10)$ \\ \hline 
Mean of $\|\hat{\bm{\alpha}}_n^{(1)} - \bm{\alpha}_0\|_\infty$ & 0.061 & 0.043 \\ 
Mean of $\|\hat{\bm{\alpha}}_n^{(1)} - \bm{\alpha}_0\|_2$ & 0.085 & 0.060
 \\ 
Proportion of $\hat{T}_n = T_0$& 0.816 & 0.946 \\
Mean of $\|\tilde{\bm{\alpha}}_{n} - \bm{\alpha}_{0}\|_\infty$ & 0.047 & 0.033 \\
Mean of $\|\tilde{\bm{\alpha}}_{n} - \bm{\alpha}_{0}\|_2$ & 0.063 & 0.043 \\
Royston test
for $\tilde{\bm{\alpha}}_{nT_0}$ & $p<0.01$ &$p=0.135$ \\
\end{tabular}
\label{sim c1 uni}
\end{center}
\end{table}
\begin{table}[!h]
\begin{center}
\small
\caption{Estimations for regression coefficients over $500$ replications based on the model \eqref{sim model uni}: Case 2.}
\begin{tabular}{c|cc}
 & $(n, p)=(1000,20)$ & $(n, p) = (2000,20)$ \\ \hline 
Mean of $\|\hat{\bm{\alpha}}_n^{(1)} - \bm{\alpha}_0\|_\infty$ & 0.064 & 0.044 \\ 
Mean of $\|\hat{\bm{\alpha}}_n^{(1)} - \bm{\alpha}_0\|_2$ & 0.090 & 0.062
 \\ 
Proportion of $\hat{T}_n = T_0$& 0.822 & 0.946 \\
Mean of $\|\tilde{\bm{\alpha}}_{n} - \bm{\alpha}_{0}\|_\infty$ & 0.048 & 0.032 \\
Mean of $\|\tilde{\bm{\alpha}}_{n} - \bm{\alpha}_{0}\|_2$ & 0.064 & 0.043 \\
Royston test
for $\tilde{\bm{\alpha}}_{nT_0}$ & $p=0.15$ &$p=0.13$ \\
\end{tabular}
\label{sim c2 uni}
\end{center}
\end{table}

\begin{table}[!h]
\begin{center}
\small
\caption{Estimations for regression coefficients over $500$ replications based on the model \eqref{sim model multi}: Case 3.}
\begin{tabular}{c|cc}
 & $(n, p)=(1000,100)$ & $(n, p) = (2000,100)$ \\ \hline 
Mean of $\|\hat{\bm{\alpha}}_n^{(1)} - \bm{\alpha}_0\|_\infty$ & 0.083 & 
0.060 \\ 
Mean of $\|\hat{\bm{\alpha}}_n^{(1)} - \bm{\alpha}_0\|_2$ & 0.114 & 0.091
 \\ 
Proportion of $\hat{T}_n = T_0$& 0.940 & 0.990 \\
Mean of $\|\tilde{\bm{\alpha}}_{n} - \bm{\alpha}_{0}\|_\infty$ & 0.057 & 0.038 \\
Mean of $\|\tilde{\bm{\alpha}}_{n} - \bm{\alpha}_{0}\|_2$ & 0.073 & 0.048 \\
Royston test
for $\tilde{\bm{\alpha}}_{nT_0}$ & $p <0.01$ &$p=0.63$ \\
\end{tabular}
\label{sim c1 multi}
\end{center}
\end{table}
\begin{table}[!h]
\begin{center}
\small
\caption{Estimations for regression coefficients over $500$ replications based on the model \eqref{sim model multi}: Case 4.}
\begin{tabular}{c|cc}
 & $(n, p)=(1000,200)$ & $(n, p) = (2000,200)$ \\ \hline 
Mean of $\|\hat{\bm{\alpha}}_n^{(1)} - \bm{\alpha}_0\|_\infty$ & 0.083 & 
0.067 \\ 
Mean of $\|\hat{\bm{\alpha}}_n^{(1)} - \bm{\alpha}_0\|_2$ & 0.115 & 0.094
 \\ 
Proportion of $\hat{T}_n = T_0$& 0.910 & 0.968 \\
Mean of $\|\tilde{\bm{\alpha}}_{n} - \bm{\alpha}_{0}\|_\infty$ & 0.056 & 0.036 \\
Mean of $\|\tilde{\bm{\alpha}}_{n} - \bm{\alpha}_{0}\|_2$ & 0.070 & 0.046 \\
Royston test
for $\tilde{\bm{\alpha}}_{nT_0}$ & $p<0.01$ &$p<0.01$ \\
\end{tabular}
\label{sim c2 multi}
\end{center}
\end{table}
\newpage
\begin{figure}[!htbp]
  \begin{center}
   \includegraphics[width=120mm]{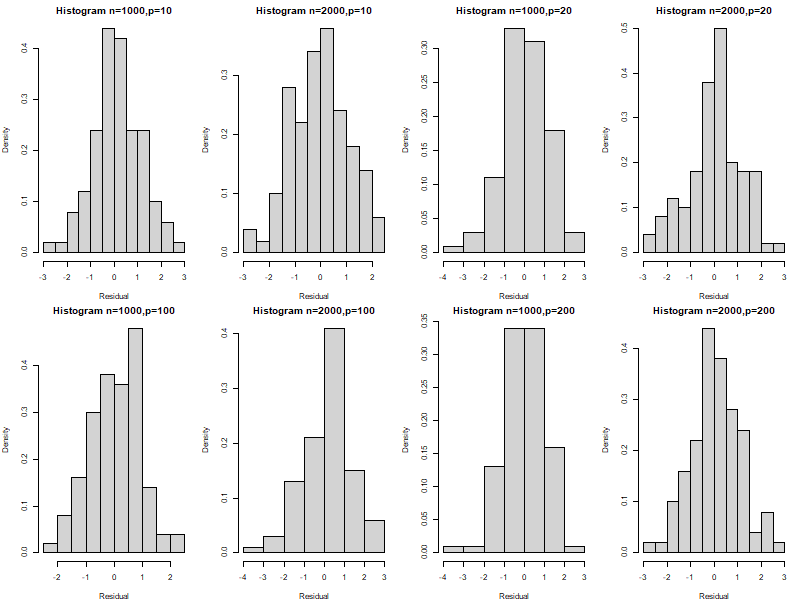}
  \end{center}
  \caption{Histograms of $\sqrt{n}\bm{u}^\top(\tilde{\bm{\alpha}}_n - \bm{\alpha}_0)$
  over 500 replications. The top four panels show the results of the Cases 1 and 2 for the model \eqref{sim model uni}, while the bottom four panels show the results of the Cases 3 and 4 for the model \eqref{sim model multi}.
}
  \label{fig:hist}
\end{figure}

\section{Applications to models of diffusion processes}\label{sec:diffusion}
\subsection{Diffusion processes with high-dimensional covariate}
Let us consider the following model of the stochastic differential equation:
\[
X_t = X_0 + \int_0^t \tilde{S}(\bm{\alpha}^\top \phi(X_s), \bm{\beta}^\top \bm{Z}_s) \dr s + \int_0^t \tilde{\sigma}(X_s, \bm{Z}_s; h) \dr W_s,\quad
t \geq 0,
\]
where $\{W_t\}_{t \geq 0}$ is a $1$-dimensional standard Brownian motion,
$\phi: \mathbb{R} \to \mathbb{R}^{p_1}$ is a measurable function,
$\{\bm{Z}_t\}_{t \geq 0}$ is a $p_2$-dimensional process, 
which is regarded as a covariate vector, 
$\bm{\alpha} \in \mathbb{R}^{p_1}$ is an 
autoregressive coefficient,
$\bm{\beta} \in \mathbb{R}^{p_2}$ is a regression coefficient for covariate process 
$\{\bm{Z}_t\}_{t \geq0}$,
$\tilde{S}$ is a measurable and twice continuously differentiable function with respect to 
$\bm{\alpha}$ and $\bm{\beta}$, and 
$\sigma$ is a measurable function.
We consider the estimation problem for 
$\bm{\theta} = (\bm{\alpha}^\top, \bm{\beta}^\top)^\top \in \mathbb{R}^{p}$, where $p= p_1 + p_2$, with the existence of a 
possibly infinite-dimensional nuisance parameter $h \in H$.
Let 
$(\bm{\theta}_0, h_0)$ be a true value of 
$(\bm{\theta}, h)$, ${\Theta} = {\Theta}_{\alpha} \times {\Theta}_{\beta} \subset \mathbb{R}^{p_1+p_2}$ 
a parameter space for $\bm{\theta}$ and $H$ a
metric space equipped with a metric $d_H$.
We observe the processes $\{X_t\}_{t \geq 0}$, 
$\{\bm{Y}_t\}_{t \geq 0}$
and $\{\bm{Z}_t\}_{t \geq 0}$
at $n+1$ equidistant time points 
$0=:t_0^n<t_1^n <\cdots<t_n^n$.
Put $T_{10}=\{j : {\alpha}_{0j} \not=0\}$, 
$T_{20} = \{p_1+j : {\beta}_{0j} \not=0\}$ and 
$T_0 = T_{10} \cup T_{20}$.
Assume that the true value $\bm{\theta}_0$ 
is high-dimensional and sparse, that is, 
$p$ possibly tends to $\infty$ as $n \to \infty$,
and the number $s=s_{1} + s_{2}$ of nonzero coefficients 
of $\bm{\theta}_0$ is relatively small, where 
$s_{1}$ is the sparsity for $\bm{\alpha}_0$ and 
$s_{2}$ is the sparsity for $\bm{\beta}_0$.
Changing the domain of $\tilde{S}$ and 
$\tilde{\sigma}$, 
we write 
\[
X_t = \int_0^t S(\bm{Y}_s; \bm{\theta}) ds
+ \int_0^t \sigma(\bm{Y}_s; \bm{\theta})
dW_s,
\]
where $\bm{Y}_s = (\phi(\bm{X}_s)^\top, \bm{Z}_s^\top)^\top$ for every $s \geq 0$.
Consider the following score function for 
$\bm{\theta}$:
\begin{eqnarray*}
\psi_n^{(1)}(\bm{\theta})
&=& \frac{1}{n \Delta_n} \sum_{k=1}^n \frac{\partial}{\partial \bm{\theta}}
S(\bm{Y}_{t_{k-1}^n}; \bm{\theta})
\{
X_{t_{k^n}}-X_{t_{k-1}^n} - S(\bm{Y}_{t_{k-1}^n}; \bm{\theta})
\Delta_n
\},
\end{eqnarray*}
where $\Delta_n = t_n^n/n=|t_k^n - t_{k-1}^n|,\ k=1,\ldots,n$.
Then, we define the following estimator for $\bm{\theta}_0$ and $T_0$, respectively;
\[
\hat{\bm{\theta}}_n^{(1)} := \arg \min_{\bm{\theta} \in \mathcal{C}_n} \|\bm{\theta}\|_1,\quad
\mathcal{C}_n := \{\bm{\theta} \in {\Theta} : \|\psi_n^{(1)}(\bm{\theta})\|_\infty \leq \lambda_n\},
\]
and 
\[
\hat{T}_n = \{j : |\hat{{\theta}}_{nj}^{(1)}| > \tau_n\},
\]
where $\tau_n$ is a threshold.

We estimate $\bm{\theta}$ based on the quasi-likelihood function given by 
\begin{eqnarray*}
L_n(\bm{\theta}, h)
&=& \prod_{i=1}^n \frac{1}{\sqrt{2 \pi \sigma^2(\bm{Y}_{t_{i-1}^n}; h)|t_i^n - t_{i-1}^n|}}
\exp\left(
- \frac{|X_{t_i^n}-X_{t_{i-1}^n}-
S(\bm{Y}_{t_{i-1}^n}; \bm{\theta})|t_i^n - t_{i-1}^n||^2}{2 \sigma^2(\bm{Y}_{t_{i-1}^n}; h)|t_i^n - t_{i-1}^n|}
\right).
\end{eqnarray*}
Let $\Psi_n(\bm{\theta}; h)$ be the derivative of 
$\log L_n(\bm{\theta}; h)/T_n$ with respect to $\bm{\theta}$, 
i.e., 
\begin{eqnarray*}
\Psi_n(\bm{\theta}, h)
&=& \frac{1}{n \Delta_n} \sum_{k=1}^n
\frac{\frac{\partial}{\partial \bm{\theta}}S(\bm{Y}_{t_{k-1}^n}; \bm{\theta})}{\sigma^2(\bm{Y}_{t_{k-1}^n}; \hat{h}_n)}
\{
X_{t_{k^n}}-X_{t_{k-1}^n} - S(\bm{Y}_{t_{k-1}^n}; \bm{\theta})\Delta_n
\}.
\end{eqnarray*}
Let $\bm{V}_n^{(1)}(\cdot)$ and $\bm{V}_n(\cdot, \cdot)$ be derivatives of 
$\psi_n^{(1)}$ and $\Psi_n$ with respect to $\bm{\theta}$, respectively, i.e., 
\begin{eqnarray*}
\bm{V}_n^{(1)}(\bm{\theta})
&=& \frac{1}{n \Delta_n} \sum_{k=1}^n \frac{\partial^2}{\partial \bm{\theta} \partial \bm{\theta}^\top}
S(\bm{Y}_{t_{k-1}^n}; \bm{\theta}) \{
X_{t_k^n} - X_{t_{k-1}^n}- S(\bm{Y}_{t_{k-1}^n}; \bm{\theta})\Delta_n
\}\\
&&- \frac{1}{n \Delta_n} \sum_{k=1}^n\frac{\partial}{\partial \bm{\theta}}S(\bm{Y}_{t_{k-1}^n}; \bm{\theta})\frac{\partial}{\partial \bm{\theta}^\top}S(\bm{Y}_{t_{k-1}^n}; \bm{\theta})\Delta_n
\end{eqnarray*}
and 
\begin{eqnarray*}
\bm{V}_n^{(2)}(\bm{\theta}, h)
&=& \frac{1}{n \Delta_n} \sum_{k=1}^n \frac{\frac{\partial^2}{\partial \bm{\theta} \partial \bm{\theta}^\top}
S(\bm{Y}_{t_{k-1}^n}; \bm{\theta})}{\sigma^2(\bm{Y}_{t_{k-1}^n}; h)} \{
X_{t_k^n} - X_{t_{k-1}^n}- S(\bm{Y}_{t_{k-1}^n}; \bm{\theta})\Delta_n
\}\\
&&- \frac{1}{n \Delta_n} \sum_{k=1}^n\frac{\frac{\partial}{\partial \bm{\theta}}S(\bm{Y}_{t_{k-1}^n}; \bm{\theta})\frac{\partial}{\partial \bm{\theta}^\top}S(\bm{Y}_{t_{k-1}^n}; \bm{\theta})}{\sigma^2(\bm{Y}_{t_{k-1}^n}; h)}\Delta_n
\end{eqnarray*}
Let $C_{T_0} := 
\{\bm{v} \in \mathbb{R}^p : \|\bm{v_}{T_0^c}\|_1 \leq \|\bm{v}_{T_0}\|_1\}$.
Hereafter, we consider the following conditions.
\begin{assumption}\label{regularity diffusion}
\begin{itemize}
\item[(i)]
For every $\bm{y}_1 = (\phi({x}_1)^\top, \bm{z}_1)$ and 
$\bm{y}_2 = (\phi({x}_2)^\top, \bm{z}_2)$, 
there exists a constant 
$L_S>0$ such that
\[
\left|
S(\bm{y}_1; \bm{\theta}_0) - S(\bm{y}_2; \bm{\theta}_0)\right|
\leq L_S s(|x_1-x_2|+\|\bm{z}_1 - \bm{z}_2\|_\infty).
\]
\item[(ii)]
The function $S$ is twice continuously differentiable with respect to 
$\bm{\theta} = (\bm{\alpha}^\top,\bm{\beta}^\top)^\top$. 
Moreover, there exists a positive constant 
$K_{p}$ possibly depending on $p$ such that 
\[
\sup_{t \geq 0} \E\left[
|S(\bm{Y}_t; \bm{\theta}_0)|^4
\right]^{1/4} \leq K_{p},
\]
\[
\sup_{t \geq 0} \E\left[
\left\|\frac{\partial}{\partial \bm{\theta}}S(\bm{Y}_t; \bm{\theta}_0)\right\|_\infty^4
\right]^{1/4} \leq K_{p},
\]
and 
\[
\sup_{t \geq 0} \E\left[ |\sigma(\bm{Y}_t; h_0)|^4
\right]^{1/4}
\leq K_p.
\]
\item[(iii)]
It holds that
\[
c:=\inf_{\bm{y} \in \mathbb{R}^{p}} \inf_{h \in H} \sigma^2(\bm{y}; h)>0.
\]
\item[(iv)]
The processes $\{X_t\}_{t \geq 0}$ and 
$\{\bm{Z}_t\}_{t \geq 0}$ satisfy that 
\[
\E\left[
\sup_{t \in [t_{i-1}^n, t_i^n] }|X_t-X_{t_{i-1}^n}|^k
\right] + \E\left[
\sup_{t \in [t_{i-1}^n, t_i^n] }\|\bm{Z}_t-\bm{Z}_{t_{i-1}^n}\|_\infty^k
\right]
\leq D_k \Delta_n^{k/2},
\]
where 
$D_k>0$ is a constant depending on $k$ and $p$.
\item[(v)]
It holds that 
\[
\Delta_n \asymp n^{-{\zeta}},\quad
\log p =o\left((n \Delta_n)^{1/2}\right)
\]
as $n \to \infty$, where $\zeta \in (1/2, 1)$.
Moreover, the tuning parameter 
${\lambda}_n$ is a sequence such that 
\[
\max\left\{\frac{K_p s\Delta_n^{1/2} D_2^{1/2}}{{\lambda}_n},
\frac{K_p^2}{{\lambda}_n}\sqrt{\frac{\log (1+p)}{n \Delta_n}}\right\} \to 0,\quad
n \to \infty.
\]
\item[(vi)]
Let $c_n := \|\psi_n^{(1)}(\hat{\bm{\theta}}_n^{(1)})-\psi_n^{(1)}(\bm{\theta}_0)\|_\infty$
and $\theta_{\min} := \inf_{j \in T_0} |\theta_{0j}|$.
Then, it holds that 
\[
\theta_{\min} > 4c_n/\delta.
\]
\end{itemize}
\end{assumption}
The condition for $\psi_n^{(1)}(\bm{\theta}_0)$
can be verified as follows.
\begin{prop}\label{prop gradient diffusion}
Under Assumption \ref{regularity diffusion}, it holds that
\begin{equation}\label{gradient diffusion}
\Prob(\|\psi_n^{(1)}(\bm{\theta}_0)\|_\infty > {\lambda}_n)
\to 0,\quad n \to \infty.
\end{equation}
\end{prop}
See Subsection \ref{subsec: proof of sec4} for the 
proof.
On the other hand, the condition for 
$\bm{V}_n^{(1)}$ strongly depends on the 
model structure, it is difficult to provide some sufficient conditions.
We thus introduce the following condition.
\begin{assumption}\label{diffusion matrix condition}
There exists a constant $\delta>0$ such that
\begin{equation}\label{matrix condition general diffusion}
\Prob\left(
\inf_{\bm{0} \not=\bm{v} \in C_{T_0}}\inf_{\bm{\theta} \in {\Theta}}
\frac{|\bm{v}^\top \bm{V}_n^{(1)}(\bm{\theta}) \bm{v}|}{\|\bm{v}_{T_0}\|_1 \| \bm{v} \|_\infty}>\delta
\right) \to 1,\quad 
n \to \infty.
\end{equation}
\end{assumption}
Note that the left-hand side of \eqref{matrix condition general diffusion} may depend on 
$n$, $p$ and $s$.
We ensure this condition for Ornstein--Uhlenbeck processes in Subsection 4.2.
The following theorem follows from Theorem \ref{rate of convergence DS general model} and Proposition \ref{prop gradient diffusion}.
\begin{thm}\label{DS error diffusion}
Let Assumptions \ref{regularity diffusion} and 
\ref{diffusion matrix condition} hold.
Then, 
it holds that
\begin{equation}\label{error bound diffusion}
\Prob\left(\|\hat{\bm{\theta}}_n^{(1)} - \bm{\theta}_0\|_\infty > \frac{2}{\delta}c_n\right)
\to 0,\quad
n \to \infty,
\end{equation}
where $c_n := \|\psi_n^{(1)}(\hat{\bm{\theta}}_n^{(1)})-\psi_n^{(1)}(\bm{\theta}_0)\|_\infty$.
Especially, it holds that 
\[
\|\hat{\bm{\theta}}_n^{(1)} - \bm{\theta}_0\|_\infty = O_p(\lambda_n),\quad n \to \infty.
\]
Moreover, if the threshold 
$\tau_n$ satisfies that 
$2 c_n / \delta< \tau_n < \theta_{\min}/2$, 
it holds that
\begin{equation}\label{selection diffusion}
\Prob(\hat{T}_n = T_0) \to 1,\quad n \to \infty.
\end{equation}
\end{thm}
As for the nuisance parameter $h$, we 
should construct a 
consistent estimator $\hat{h}_n$.
As well as the previous section, 
we assume the following condition.
\begin{assumption}\label{sparsity variance diffusion}
The function $\sigma$ satisfies that 
\[
\sigma(\bm{y}; h)
= \sigma({\bm{y}}_{(T_0)}; h),\quad
\forall \bm{y} \in \mathbb{R}^p, h \in H.
\]
\end{assumption}
Under such condition, since 
the 
variable selection by $\hat{T}_n$ is possible, 
we construct consistent estimators 
by the classical methods for finite-dimensional cases.
When $H$ is a finite-dimensional space, see, e.g., \cite{yoshida1992estimation, kessler1997estimation}
while see, e.g., \cite{hoffmann2001estimating} for the case where $H$ is an infinite-dimensional space
for the consistent estimation.

Then, we define the estimator $\tilde{\bm{\theta}}_n$ for $\bm{\theta}$ as a solution to the following equations:
\[
\Psi_{n\hat{T}_n}(\bm{\theta}_{\hat{T}_n}, \hat{h}_n) = \bm{0},\quad
\bm{\theta}_{\hat{T}_n^c} = \bm{0},
\] 
where for every index set $T = T_1 \cup T_2$, 
\begin{eqnarray*}
\Psi_{n T}(\bm{\theta}_{T}, \hat{h}_n)
&=& \frac{1}{n \Delta_n} \sum_{k=1}^n
\frac{\frac{\partial}{\partial \bm{\theta}_T}S(\bm{Y}_{t_{k-1}^n (T)}; \bm{\theta}_{(T)})}{\sigma^2(\bm{Y}_{t_{k-1}^n (T)}; \hat{h}_n)}\cdot\{
X_{t_{k^n}}-X_{t_{k-1}^n} - S(\bm{Y}_{t_{k-1}^n (T)}; \bm{\theta}_{(T)})\Delta_n
\}.
\end{eqnarray*}
Using $\hat{T}_n$, we can define the estimator
$\tilde{\bm{\theta}}_n$ as a solution to the following equations;
\[
\Psi_{n \hat{T}_n}(\bm{\theta}_{\hat{T}_n}, \hat{h}_n) = \bm{0},\quad
\bm{\theta}_{\hat{T}_n^c} = \bm{0}.
\]
We assume the following conditions in order to establish the 
asymptotic behavior of the estimator $\tilde{\bm{\theta}}_n$.
\begin{assumption}\label{2nd estimator diffusion}
\begin{itemize}
\item[(i)]
The sparsity $s$ is independent of $n$ and $p$.
The parameter space ${\Theta}_{T_0} \subset \mathbb{R}^s$ is compact and the true value $\bm{\theta}_{0 T_0}$ is an interior point of ${\Theta}_{T_0}$.
There exists a measurable function 
$\Lambda_1$ on $\mathbb{R} \times \mathbb{R}^{s_2}$ such that 
\begin{eqnarray*}
\lefteqn{
S(\bm{y}_{(T_0)};\bm{\theta}) - 
S(\bm{y}_{(T_0)};\bm{\theta}_0)}\\
&=& \frac{\partial}{\partial \bm{\theta}_{T_0}}S(\bm{y}_{(T_0)};\bm{\theta}_0)
(\bm{\theta}_{T_0} - \bm{\theta}_{0 T_0})
+ \Lambda_1({\bm{y}}_{T_0} ) \epsilon(x, \bm{z}_{T_{20}}; \bm{\theta}_{T_0}, \bm{\theta}_{0 T_0})
\end{eqnarray*}
for every 
$\bm{y} = (\phi(x)^\top, \bm{z})$
and $\bm{\theta} \in {\Theta}$ where $\epsilon$ is a 
measurable function satisfying that 
\[
\sup_{x \in \mathbb{R}, \bm{z}_{T_{20}} \in \mathbb{R}^{s_2}}
|\epsilon(x, \bm{z}_{T_{20}}; \bm{\theta}_{T_0}, \bm{\theta}_{0 T_0})|
= o(\|\bm{\theta}_{T_0}-\bm{\theta}_{0 T_0}\|_2),\quad
\bm{\theta}_{T_0} \to \bm{\theta}_{0 T_0}.
\]
\item[(ii)]
There exists a constant $K_{T_0}$ such that 
\[
\sup_{\bm{\theta}_{T_0} \in {\Theta}_{T_0}}
|S(\bm{y}_{(1T_0}; \bm{\theta}_{(T_0)})
-S(\bm{y}_{(2T_0}; \bm{\theta}_{(T_0)})|
\leq K_{T_0} \|{\bm{y}}_{1T_0}- {\bm{y}}_{2T_0}\|_2;
\]
\[
\sup_{\bm{\theta}_{T_0} \in {\Theta}_{T_0}}
\left\|\frac{\partial}{\partial \bm{\theta}_{T_0}}S(\bm{y}_{(1T_0}; \bm{\theta}_{(T_0)})
-\frac{\partial}{\partial \bm{\theta}_{T_0}}S(\bm{y}_{(2T_0}; \bm{\theta}_{(T_0)})\right\|
\leq K_{T_0} \| \bm{y}_{1T_0} - \bm{y}_{2T_0}\|_2;
\]
\[
\sup_{h \in H}
|\sigma^2(\bm{y}_{1(T_0)}; h) - \sigma^2(\bm{y}_{2(T_0)}; h)|
\leq K_{T_0}\|\bm{y}_{1T_0} - \bm{y}_{2T_0}\|_2.
\]
\item[(iii)]
There exists a measurable function $\Lambda_2$ such that 
\begin{eqnarray*}
\lefteqn{
\left\|\frac{\partial}{\partial \bm{\theta}_{T_0}}S(\bm{y}_{(T_0)}; \bm{\theta}_{1(T_0)})
- \frac{\partial}{\partial \bm{\theta}_{T_0}}S(\bm{y}_{(T_0)}; \bm{\theta}_{2(T_0)})\right\|_2}\\
&\leq& \Lambda_2(\bm{y}_{T_0})\|\bm{\theta}_{1 T_0} - \bm{\theta}_{2 T_0}\|_2,\quad
\forall \bm{\theta}_{1 T_0}, \bm{\theta}_{2 T_0} \in {\Theta}_{T_0};
\end{eqnarray*}
\[
|\sigma^2(\bm{y}_{(T_0)}; h_1)-\sigma^2(\bm{y}_{(T_0)}; h_2)|
\leq \Lambda_2(\bm{y}_{T_{0}})d_H(h_1, h_2),\quad
\forall h_1, h_2 \in H ;
\]
for every $\bm{y} = (\phi(x)^\top,\bm{z}^\top)^\top$.
\item[(iv)]
It holds that 
\[
\sup_{t \geq 0}
\E\left[
|\Lambda_1(\bm{Y}_{t T_0})|^8 +  |\Lambda_2(\bm{Y}_{t T_0})|^8 +|X_t|^4
\right] < \infty.
\]
\item[(v)]
The process $\{\bm{Y}_{t T_0}\}_{t \geq 0}$
is ergodic with an invariant measure $\mu_0$, 
i.e., for every $\mu_0$-integrable function $g$, 
it holds that 
\[
\frac{1}{T} \int_0^T g(\bm{Y}_{t T_0}) \dr t
\to^p \int_{\mathbb{R}^{s}}
g(\bm{y}_{T_{0}}) \mu_0(\dr \bm{y}_{T_{0}}),\quad
T \to \infty.
\]
\item[(vi)]
The following $s \times s$ matrix $\bm{\mathcal{I}}(\bm{\theta}_0, h_0)$ is positive definite:
\[
\bm{\mathcal{I}}(\bm{\theta}_0, h_0)
= \int_{\mathbb{R}^{s}}
\frac{\frac{\partial}{\partial \bm{\theta}_{T_0}}S(\bm{y}_{(T_0)}; \bm{\theta}_{0 T_0}) \frac{\partial}{\partial \bm{\theta}_{T_0}^\top} S(\bm{y}_{(T_0)}; \bm{\theta}_{0 T_0})}{\sigma^2(\bm{y}_{(T_0)}; h_0)} \mu_0(\dr \bm{y}_{T_0}).
\]
\item[(vii)]
It holds that 
\begin{eqnarray*}
\inf_{\bm{\theta}_{T_0}: \|\bm{\theta}_{T_0}-\bm{\theta}_{0 T_0}\|_2>\epsilon}
\left\|
\int_{\mathbb{R} \times \mathbb{R}^{s_2}}
\frac{\frac{\partial}{\partial \bm{\theta}_{T_0}}S(\bm{y}_{(T_0)}; \bm{\theta}_{(T_0)})}{\sigma^2(\bm{y}_{(T_0)}; h_0)}
\left[
S(\bm{y}_{(T_0)};\bm{\theta}_{(T_0)})
-S(\bm{y}_{(T_0)};\bm{\theta}_{0(T_0)})
\right]\mu(\dr \bm{y}_{T_{0}})
\right\|_2
> 0
\end{eqnarray*}
for every $\epsilon >0$.
\item[(viii)]
The metric entropy condition for $(H, d_H)$ is satisfied:
\[
\int_0^1 \sqrt{\log N(H, d_H, \epsilon)} \dr \epsilon < \infty.
\]
\item[(ix)]
The estimator $\hat{h}_n$ is a consistent estimator of 
$h_0$, i.e., 
$d_H(\hat{h}_n, h_0) = o_p(1)$ as $n \to \infty$.
\end{itemize}
\end{assumption}
Note that Assumption \ref{2nd estimator diffusion} 
corresponds to conditions A1-A8 of \cite{nishiyama2009asymptotic}.
The sufficient condition for (ix) is described in 
Theorem 4.4 of \cite{nishiyama2009asymptotic}
when $\hat{h}_n$ is a minimizer of
\begin{eqnarray*}
\mathcal{A}_n(h)
= \frac{1}{n} \sum_{k=1}^n \left|
\frac{|X_{t_{k}^n} - X_{t_{k-1}^n}|^2}{\Delta_n} - \sigma^2(\bm{Y}_{t_{k-1}^n}; h)
\right|^2.
\end{eqnarray*}
Finally, we obtain the following results.
\begin{lem}\label{normality diffusion}
Under Assumptions \ref{regularity diffusion}, 
\ref{diffusion matrix condition}, 
\ref{sparsity variance diffusion}, and \ref{2nd estimator diffusion},
it holds that 
\[
\sqrt{n \Delta_n} \left(\tilde{\bm{\theta}}_{n \hat{T}_n} - \bm{\theta}_{0 T_0}\right) 1_{\left\{\hat{T}_n = T_0\right\}}
\to^d \mathcal{N}_s(0, \bm{\mathcal{I}}(\bm{\theta}_0, h_0)^{-1}),\quad
n \to \infty.
\]
\end{lem}
\begin{proof}
It follows from Lemma 4.1 and Theorem 4.3 
of \cite{nishiyama2009asymptotic} that 
\[
\sqrt{n\Delta_n} \left(\tilde{\bm{\theta}}_{n {T_0}} - \bm{\theta}_{0 T_0}\right) \to^d \mathcal{N}_s(0, \bm{\mathcal{I}}(\bm{\theta}_0, h_0)^{-1}),\quad
n \to \infty.
\]
Combining this fact with Theorem
\ref{DS error diffusion}, we obtain the conclusion
as well as Theorem \ref{normality general model}.
\end{proof}
The weak convergence of the proposed
estimator in  the Hilbert space $\ell^2$ 
can be shown as well as Section \ref{sec:general}.
Consider the $\ell^2$-valued random sequence 
$\{\bm{\mathcal{R}}_n\}_{n \in \mathbb{N}}$ given by 
\[
\langle \bm{e}_j, \bm{\mathcal{R}}_n\rangle
= \begin{cases}
\sqrt{n \Delta_n} (\tilde{\bm{\theta}}_{nj} - \bm{\theta}_{0j} ) & 
1 \leq j \leq p \\
0 & j > p
\end{cases}.
\] 
\begin{lem}\label{cramer-wold diffusion}
Under the same assumptions as Lemma \ref{normality diffusion}, 
it holds that 
\[
\langle \bm{u}, \bm{\mathcal{R}_n} \rangle
\to^d \mathcal{N}(0, \bm{u}_{T_0}^\top 
\bm{I}(\bm{\theta}_0, h_0)^{-1}\bm{u}_{T_0}),\quad
n \to \infty,
\]
for every $\bm{u} \in \ell^2$.
\end{lem}
\begin{thm}\label{weak convergence l2 diffusion}
Under the same assumptions as Lemma \ref{normality diffusion}, it holds that 
\[
\bm{\mathcal{R}}_n \to^d \bm{\mathcal{R}},\quad n \to \infty \quad
\mbox{in}\ \ell^2,
\]
where 
$\bm{\mathcal{R}}$ is a centered Gaussian random field whose 
distribution is determined by 
the distribution of its marginals 
\[
\langle \bm{u}, \bm{\mathcal{R}} \rangle \sim
\mathcal{N}(0, \bm{u}_{T_0}^\top\bm{I}(\bm{\theta}_0, h_0)^{-1}\bm{u}_{T_0})
\]
for every $\bm{u} \in \ell^2$.
\end{thm}
Since Lemma \ref{cramer-wold diffusion}
and Theorem \ref{weak convergence l2 diffusion} can be proved as well as 
Lemma \ref{cramer-wold general}
and Theorem 
\ref{weak convergence l2 general}, respectively,
we omit their proofs.
\subsection{Example: A linear model of 
Gaussian diffusion processes}
Let us consider the $(p+1)$-dimensional
Ornstein--Uhlenbeck
process $\{\bm{Y}_t\}$ which is a solution to the following
equation:
\[
\bm{Y}_t = \bm{Y}_0 + \int_0^t \bm{A} \bm{Y}_s \dr s + \bm{\Sigma} \bm{W}_t,\quad t \geq 0,
\]
where $\bm{A}, \bm{\Sigma} \in \mathbb{R}^{(p+1)\times (p+1)}$
are unknown matrices and 
$\{\bm{W}_t\}_{t \geq 0}$ is a $(p+1)$-dimensional 
standard Brownian motion.
Put $\bm{Y}_t = (X_t, \bm{Z}_t^\top)^\top$,
$\bm{W}_t = (W_t, W_t^1,\ldots,W_t^{p})^\top$ 
for every $t \geq 0$.
When $\bm{\Sigma} = \diag(\sigma,\sigma_1,\ldots,\sigma_p)$, it holds that 
\[
X_t = X_0 + \int_0^t \{{\alpha} X_s + \bm{\beta}^\top \bm{Z}_s\} \dr s + \sigma W_t,\quad 
t \geq0,
\]
where $\bm{\theta}^\top :=({\alpha}, \bm{\beta}^\top)$ is the first row of 
the matrix $\bm{A}$.
Let $\bm{\theta}_0^\top = ({\alpha}_0, \bm{\beta}_0^\top)$ 
and $\sigma_0$ be true values of $\bm{\theta}$ and 
$\sigma$.
We can consider the following 
score function and Hessian matrix:
\[
\psi_n^{(1)}(\bm{\theta}) = \frac{1}{n \Delta_n} \sum_{k=1}^n \bm{Y}_{t_{k-1}^n} \left\{
X_{t_k^n} - X_{t_{k-1}^n} - \bm{\theta}^\top \bm{Y}_{t_{k-1}^n}
\right\},\quad
\bm{V}_n^{(1)} = -\frac{1}{n \Delta_n}\sum_{k=1}^n \bm{Y}_{t_{k-1}^n} \bm{Y}_{t_{k-1}^n}^\top \Delta_n.
\]
We assume the following conditions.
\begin{assumption}\label{regularity linear diffusion}
\begin{itemize}
\item[(i)]
The matrix $\bm{A}$ is diagonalizable and the real parts of its eigenvalues $\varsigma_p \leq \ldots \leq \varsigma_1 $ are all negative.
Moreover, the initial value 
$\bm{Y}_0 \sim \mathcal{N}_{p+1}( \bm{0} , \bm{V})$.
\item[(ii)]
There exist positive constants 
$f_\infty$ and $K_p$,
\[
F_\infty(T_0, \bm{V})>f_\infty,\quad
 \|\bm{V}^* \bm{V}\|_2 \leq K_p,
\]
where $\bm{V}^*$
is the Hermitian adjoint of $\bm{V}$.
\item[(iii)]
The matrix 
$\bm{\mathcal{I}}(\bm{\theta}_0, \sigma_0)= 
\bm{V}_{T_0, T_0}/\sigma_0^2$ is invertible.
\end{itemize}
\end{assumption}
Condition (i) implies that 
$\{\bm{Y}_t\}$ is strictly stationary Gaussian process with covariance function; 
\[
C(s, t):= \Cov(\bm{Y}_s, \bm{Y}_t)
= \begin{cases}
\exp((s-t)\bm{A})\bm{V} & 0 \leq t \leq s < \infty \\
\bm{V}\exp((t-s)\bm{A}^\top) & 0 \leq s < t < \infty. \\
\end{cases}
\]
See Chapter 5.6 of \cite{karatzas1991brownian} 
for details of condition (i).
Especially, it holds that 
\[
\sum_{l=0}^\infty \exp(l \Delta_n \varsigma_1) < \infty
\]
under Assumption \ref{regularity linear diffusion}.
Then, we have the following result.
\begin{prop}\label{linear diffusion maximal ineq}
Let Assumption \ref{regularity linear diffusion} hold.
\begin{itemize}
\item[(i)]
It holds that 
\[
\E[\|\bm{Y}_0\|_\infty^4]^{1/4} < K_p.
\]
\item[(ii)]
There exists a constant $\eta \in (0, f_\infty/2)$ such that 
\[
\Prob\left(
F_\infty(T_0, \bm{V}_n^{(1)}) > f_\infty - 2 \eta
\right) \to 1,\quad n \to \infty.
\]
\end{itemize}
\end{prop}
See Subsection \ref{subsec: proof of sec4} for the 
proof.
The assertion (i) implies 
\[
\Prob(\|\psi_n^{(1)}(\bm{\theta}_0)\|_\infty> \lambda_n) \to 0,\quad
n \to \infty,
\]
when $\lambda_n$ satisfies that
\[
\max\left\{\frac{K_p \Delta_n^{1/2} D_2^{1/2}}{{\lambda}_n},
\frac{K_p^2}{{\lambda}_n}\sqrt{\frac{\log (1+p)}{n \Delta_n}}\right\} \to 0,\quad
n \to \infty.
\]
Note that the process $\{\bm{Y}_{t T_0}\}$ is ergodic 
with an invariant measure under Condition 
(i) of Assumption \ref{regularity linear diffusion}.
As for the diffusion coefficient $\sigma$, 
we can construct a consistent estimator 
by classical methods, see, e.g., \cite{yoshida1992estimation}. 
In conclusion, under Assumption \ref{regularity linear diffusion}, we can apply the estimation procedure 
described in the previous subsection and obtain an 
asymptotically normal estimator.
\section{Proofs}\label{sec:proof}
\subsection{Stochastic maximal inequalities and concentration inequalities}\label{subsec: maximal ineqs}
In this section, we introduce three 
useful lemmas which are used to derive
the error bounds of the Dantzig selector.

The following maximal inequality for square integrable martingales is provided in \cite{nishiyama2021martingale}.
\begin{lem}[Theorem A1.1.6 of \cite{nishiyama2021martingale}]\label{maximal ineq}
Let $p$ be any positive integer.
Let a $p$-dimensional martingale difference sequence $\bm{\xi} = (\xi^1,\ldots,\xi^p)^\top$
on a stochastic basis $\bm{B} = (\Omega, \mathcal{F}, \Prob, \{\mathcal{F}_k\}_{k \geq 0})$ such that 
$\E[(\xi_k^i)^2] < \infty$ for all $i, k$ given.
Then, it holds for any finite stopping time 
$T$ that 
\[
\E\left[
\max_{1 \leq i \leq p} \left|
\sum_{k=1}^T \xi_k^i
\right|
\right]
\leq 2\sqrt{2} \sqrt{\E\left[\sum_{k=1}^T \max_{1 \leq i \leq p} (\xi_k^i)^2\right]}\sqrt{\log (1+p)}
\]
and that 
\[
\E\left[
\max_{1 \leq i \leq p} \left|
\sum_{k=1}^T \xi_k^i
\right|
\right]
\leq 2\sqrt{2} \sqrt{\E\left[\sum_{k=1}^T \E\left[\max_{1 \leq i \leq p} (\xi_k^i)^2|\mathcal{F}_{k-1}\right]\right]}\sqrt{\log (1+p)}.
\]
\end{lem}
The inequality given by the following lemma
is a variant of the Hanson--Wright inequality for 
Gaussian random variables, provided by \cite{wong2020lasso}.
See also \cite{rudelson2013hanson} for details.
\begin{lem}[Lemma 11 of \cite{wong2020lasso}]\label{hanson-wright}
Let $\bm{Y} \sim \mathcal{N}_n(\bm{0}, \bm{Q})$
be an $n$-dimensional random vector.
Then, there exists a universal constant 
$c>0$ such that for any $\eta>0$,
\[
\Prob\left(
\frac{1}{n}\left|
\|\bm{Y}\|_2^2-\E[\|\bm{Y}\|_2^2]
\right|
> \eta \| \bm{Q} \|_2
\right)
\leq 2 \exp\left(
-cn \min \{\eta, \eta^2\}
\right).
\]
\end{lem}
The last lemma gives a concentration inequality for 
non-Gaussian time series satisfying
geometric ${\tau}$-mixing condition, which is essentially equivalent to the Theorem 1 of \cite{merlevede2011bernstein}.
See also \cite{wong2020lasso} for the proof.
\begin{lem}\label{sub-weibull ineq revised}
Let $\{X_t\}_{t \in \mathbb{Z}}$ be an 
$\mathbb{R}$-valued $\tau$-mixing zero mean strictly stationary process
which satisfies that 
\[
\tau(l) \leq a \exp(-c l^{\gamma_1}),\quad
l \in \mathbb{N}
\]
and
\[
\sup_{q \geq 1}\E[|X_t|^{q}]^{1/q}{q}^{-1/\gamma_2} \leq K,\quad \forall t \in \mathbb{Z}
\]
for some constants $a, c, \gamma_1, \gamma_2$, and 
$K>0$.
Let 
\[
\gamma := \left(
\frac{1}{\gamma_1} + \frac{1}{\gamma_2}
\right)^{-1} < 1.
\]
Then, for every $n >4$ and $\epsilon > 1/n$,
it holds that 
\begin{eqnarray}\label{concentration:merlevede}
\Prob\left(
\left|
\frac{1}{n} \sum_{t=1}^n X_t
\right| > \epsilon
\right)
&\leq& 
n \exp\left(
-\frac{(\epsilon n)^\gamma}{K^\gamma C_1}
\right)
+ \exp\left(
- \frac{\epsilon^2 n}{K^2C_2\{1+n^2 \mathcal{V}\}}
\right) \nonumber \\
&& + \exp\left(
-\frac{\epsilon^2 n}{K^2 C_3}
\right),
\end{eqnarray}
where 
\begin{equation}\label{def:V}
\mathcal{V} = \sup_{M>0} \sup_{i >0} \left(
\Var[\varphi_M(X_i)] + 2 \sum_{j > i} |\Cov(\varphi_M(X_i), \varphi_M(X_j))|
\right),
\end{equation}
with 
\[
\varphi_M(x) = (x \land M) \lor (-M),\quad
x \in \mathbb{R},\quad M >0
\]
and
$C_k, k=1,2,3$ are constants depending only on $a$, $c$, $\gamma_1$, and $\gamma_2$.
\end{lem}

\subsection{Proofs for Section 3}\label{subsec: proof of sec3}
\begin{proof}[Proof of Proposition \ref{gradient time series}]
Let us first note that 
\[
\psi_n^{(1)}(\bm{\theta}_0)
= \frac{1}{n}\sum_{t=1}^n \frac{\partial}{\partial \bm{\theta}} S(\bm{Y}_{t-1}; \bm{\theta}_0)u_t.
\]
It follows from Lemma \ref{maximal ineq} that 
\begin{eqnarray*}
\E\left[
\|\psi_n^{(1)}(\bm{\theta}_0)\|_\infty
\right]
&=&  \E\left[
\max_{1 \leq i \leq p} \left|
\sum_{t=1}^n \xi_t^i
\right|
\right]\\
&\leq& 2\sqrt{2} \sqrt{\E\left[\sum_{t=1}^n \max_{1 \leq i \leq p} (\xi_t^i)^2\right]}\sqrt{\log (1+p)},
\end{eqnarray*}
where
\[
\xi_t^i
= \frac{1}{n}\frac{\partial}{\partial \bm{\theta}_i} S(\bm{Y}_{t-1}; \bm{\theta}_0)u_t.
\]
We have 
\begin{eqnarray*}
\E\left[\sum_{t=1}^n \max_{1 \leq i \leq p} (\xi_t^i)^2\right]
&=&\frac{1}{n^2} \sum_{t=1}^n 
\E\left[
\left\|
\frac{\partial}{\partial \bm{\theta}} S(\bm{Y}_{t-1}; \bm{\theta}_0)
\right\|_\infty^2
\E[u_t^2|\mathcal{F}_{t-1}]
\right] \\
&=&\frac{1}{n^2} \sum_{t=1}^n \E\left[
\left\|
\frac{\partial}{\partial \bm{\theta}} S(\bm{Y}_{t-1}; \bm{\theta}_0)
\right\|_\infty^2
\sigma^2(\bm{Y}_{t-1}; h_0)
\right] \\
&\leq& \frac{K_{p, s}}{n}.
\end{eqnarray*}
Therefore, we obtain 
\[
\E\left[
\|\psi_n^{(1)}(\bm{\theta}_0)\|_\infty
\right]
\leq 2 \sqrt{2}K_{p, s} \sqrt{\frac{\log(1+p)}{n}},
\]
which ends the proof.
\end{proof}
\begin{proof}[Proof of Proposition \ref{INAR maximal ineq}]
\begin{itemize}
\item[(i)]
Under Condition (iv), it follows from Theorem 1 of \cite{ahle2022sharp} that
\[
\E[X_t^q | \mathcal{F}_{t-1}]
\leq (\bm{\theta}^\top \bm{Y}_{t-1} + q/2)^q,
\]
hence that 
\[
\E[X_t^q]
\leq \E[(\bm{\theta}^\top \bm{Y}_{t-1} + q/2)^q].
\]
Using the triangle inequality, we have 
\begin{eqnarray*}
\|X_t\|_{L^q}
&\leq& \|\bm{\theta}^\top \bm{Y}_{t-1} + q/2\|_{L^q}\\
&\leq& \left\|\mu_\epsilon + \sum_{i=1}^p {\alpha}_i X_{t-i}\right\|_{L^q} + q/2\\
&\leq& \mu_\epsilon + \sum_{i=1}^p {\alpha}_i \|X_{t-i}\|_{L^q} + q/2.
\end{eqnarray*}
Since $\{X_t\}_{t \in \mathbb{Z}}$ is a stationary process, we have
\[
\|X_t\|_{L^q}
\leq q K_X.
\]
As well as the proof of Proposition \ref{gradient time series}, we have 
\begin{eqnarray*}
\Prob(\|\psi_n^{(1)}(\bm{\theta}_0)\|_\infty)
&\leq& \frac{2 \sqrt{2}}{\lambda_n}
\sqrt{\sum_{t=1}^n \E\left[\max_{1 \leq i \leq p}|\xi_t^i|^2\right]}\sqrt{\log(1+p)},
\end{eqnarray*}
where 
\[
\xi_t^i = \frac{1}{n}\bm{Y}_{t-1}^i u_t.
\]
Since the sparsity of $h_0$ is at most
$s$, we have 
\begin{eqnarray*}
\E\left[\max_{1 \leq i \leq p}|\xi_t^i|^2\right]
&=&\frac{1}{n^2}\E\left[
\|\bm{Y}_{t-1}\|_\infty^2 \E[u_t^2 | \mathcal{F}_{t-1}]
\right]\\
&=& \frac{1}{n^2}\E\left[
\|\bm{Y}_{t-1}\|_\infty^2 \sigma^2(\bm{X}_{t-1}; \bm{h}_0)
\right]\\
&=& \frac{1}{n^2}\E\left[
\|\bm{Y}_{t-1}\|_\infty^2 (\bm{h}_0^\top \bm{Y}_{t-1})
\right]\\
&\leq&\frac{1}{n^2}\E\left[
\|\bm{Y}_{t-1}\|_\infty^2 \|\bm{h}_0\|_1 \|\bm{Y}_{t-1}\|_\infty)
\right]\\
&\leq& \frac{s \|\bm{h}_0\|_\infty}{n^2}\E\left[
\|\bm{Y}_{t-1}\|_\infty^3
\right]\\
\end{eqnarray*}
Since $\{X_t\}_{t \in \mathbb{Z}}$ is a stationary process, we have
\begin{eqnarray*}
\E\left[
\|\bm{Y}_{t-1}\|_\infty^3
\right]
&=& \E\left[\left\{\max_{1 \leq i \leq p}|X_{t-i}|+1\right\}^3\right] \\
&\leq& 4 \E\left[ \max_{1 \leq i \leq p}
|X_{t-i}|^3\right] + 4 \\
&\leq& 4 p \E[|X_0|^3] + 4\\
&\leq& 4p \cdot 3^3 K_X^3 + 4, 
\end{eqnarray*}
which implies the conclusion.
\item[(ii)]
For every $\bm{v}=(v_0, v_1,\ldots,v_p)^\top \in C_{T_0} \cap \mathbb{S}_p$, 
note that
\begin{eqnarray*}
\|\bm{v}\|_1
&\leq& 2\|\bm{v}_{T_0}\|_1\\
&\leq& 2 \sqrt{s} \|\bm{v}_{T_0}\|_2\\
&\leq& 2 \sqrt{s}.
\end{eqnarray*}
Therefore, we have
\begin{eqnarray*}
\|\bm{v}^\top \bm{Y}_{t-1}\|_{L^q}
&\leq& \left\|v_0 + \sum_{j=1}^p v_j X_{t-j}\right\|_{L^q} \\
&\leq& |v_0| + \sum_{j=1}^p |v_j| q K_X\\
&\leq& \|\bm{v}\|_1q(1+K_X) \\
&\leq& 2 \sqrt{s} \| \bm{v} \|_2 q(1+K_X) \\
&\leq& 2 \sqrt{s} q (1+K_X),
\end{eqnarray*}
which completes the proof for $K_s = 2\sqrt{s}(1+K_X)$.
\item[(iii)]
Let $\bm{W}_n=\bm{V}_n^{(1)}-\bm{V}$.
Since we have
\begin{eqnarray*}
F_\infty(T_0, \bm{V}_n^{(1)})
&=&\inf_{\bm{v} \in C_{T_0}\setminus\{ \bm{0}\} } \frac{|\bm{v}^\top \bm{V}_n^{(1)}\bm{v}|}{\|\bm{v}_{T_0}\|_1 \|\bm{v}\|_\infty}\\
&=&  \inf_{\bm{v} \in C_{T_0}\setminus\{ \bm{0} \}} \frac{|\bm{v}^\top \bm{V} \bm{v}|}{\|\bm{v}_{T_0}\|_1 \|\bm{v}\|_\infty}
- \sup_{\bm{v} \in C_{T_0}\setminus \{ \bm{0}\} } \frac{|\bm{v}^\top \bm{W}_n \bm{v}|}{\|\bm{v}_{T_0}\|_1 \|\bm{v}\|_\infty},
\end{eqnarray*}
it suffices to show that 
for every $0 < \eta< f_\infty/2$, 
\begin{equation}\label{Wn inequality INAR}
\Prob\left(
\sup_{\bm{v} \in C_{T_0}\cap \mathbb{S}_p}
|\bm{v}^\top \bm{W}_n \bm{v}| > \eta
\right) \to 0,\quad
n \to \infty.
\end{equation}

For every $\bm{v} \in C_{T_0} \cap \mathbb{S}_p$, define that 
\[
\bm{Z}_t(v) := |\bm{v}^\top \bm{Y}_{t-1}|^2 - \E[|\bm{v}^\top \bm{Y}_{t-1}|^2].
\]
Then, we can observe that 
\[
\bm{v}^\top \bm{W}_n \bm{v} 
= \frac{1}{n}\sum_{t=1}^n \bm{Z}_t(\bm{v}).
\]
Since $\bm{v} ^\top \bm{Y}_{t-1}$ satisfies that 
\[
\sup_{q \geq 1}\E[|\bm{v}^\top \bm{Y}_{t-1}|^q]^{1/q}q^{-1} \leq K_s,
\]
we have 
\begin{eqnarray*}
\|\bm{Z}_{t}(\bm{v})\|_{L^q}
&=& \left\||\bm{v}^\top \bm{Y}_{t-1}|^2 - \E[|\bm{v}^\top \bm{Y}_{t-1}|^2]\right\|_{L^q} \\
&\leq& \left\||\bm{v}^\top \bm{Y}_{t-1}|^2\right\|_{L^q} + \E[|\bm{v}^\top \bm{Y}_{t-1}|^2] \\
&=& \E [|\bm{v}^\top \bm{Y}_{t-1}|^{2q}]^{1/q} + 4 K_s^2 \\
&=& \left(\E [|\bm{v}^\top \bm{Y}_{t-1}|^{2q}]^{1/2q}\right)^{2}+4 K_s^2 \\
&\leq& 4 q^2 K_s^2 + 4 K_s^2,\quad q \geq 1,
\end{eqnarray*}
which implies that there exists a constant $\tilde{K} > 0$ such that 
\[
\sup_{q \geq 1}\E[|\bm{Z}_t(\bm{v})|^{q}]^{1/q}{q}^{-2} \leq \tilde{K} s.
\]
By Lemma \ref{beta-mixing INAR}, 
we have 
\[
{\beta}(l) \leq as \exp(-cl^{\gamma_1}),\quad l \in \mathbb{N}.
\]
Then, as well as \cite{wong2020lasso}, 
we can evaluate the constant $\mathcal{V}$ 
appeared in 
\eqref{concentration:merlevede} as follows.
\begin{eqnarray*}
\mathcal{V}
&\leq& \|\bm{Z}_1(\bm{v})\|_{L^2}^2 + 4\|\bm{Z}_1(\bm{v})\|_{L^4}^2 \sum_{k=0}^\infty \sqrt{{\beta}(k)}\\
&\leq& \|\bm{Z}_1(\bm{v})\|_{L^2}^2 + 4\|\bm{Z}_1(\bm{v})\|_{L^4}^2 
\sqrt{as}\sum_{k=0}^\infty 
\exp\left(-\frac{1}{2} ck^{\gamma_1}\right) \\ 
&\leq& \tilde{C}_2 s^{5/2} 
\end{eqnarray*}
for some constant $\tilde{C}_2>0$ depending on 
$a, c$, and $\tilde{K}$.
It follows from Lemma \ref{sub-weibull ineq revised} that 
\begin{eqnarray*}
\Prob\left(
\left|
\frac{1}{n}\sum_{t=1}^n \bm{Z}_t(\bm{v})
\right|>\eta
\right)
&\leq& n \exp\left(
-\frac{(\eta n)^{\gamma}}{\tilde{K}^{\gamma}s^{\gamma}C_1}
\right)
+ \exp\left(
-\frac{\eta^2 n^2}{\tilde{K}^2 s^2 C_2\{1 + n \tilde{K}^2 s^{5/2}\}}
\right) \\
&& +\exp\left(
- \frac{\eta^2 n}{\tilde{K}^2 s^2 C_3}
\right)
\end{eqnarray*}
where 
\[
\gamma = \left(
\frac{1}{\gamma_1} + 2
\right)^{-1}
= \frac{\gamma_1}{1+2\gamma_1},
\]
$C_1$ and $C_2$ are positive 
constants depending only on $c$ appeared in the mixing condition (iii) of Assumption 3.8.
Therefore, as well as Proof of Proposition 3.1 in \cite{basu2015regularized} and the $\epsilon$-net argument provided in \cite{vershynin2018high}, 
we have
\begin{eqnarray*}
\Prob \left(
\sup_{\bm{v} \in C_{T_0}\cap \mathbb{S}_p} |\bm{v}^\top \bm{W}_n \bm{v}|
> \eta
\right)
&\leq&
\Prob \left(
9\sup_{v \in \mathcal{B}_0(2s)} |\bm{v}^\top \bm{W}_n \bm{v}|
> \eta
\right) \\
&\leq&
\Prob \left(
\sup_{\bm{v} \in \mathcal{B}_0(2s)} |\bm{v}^\top \bm{W}_n \bm{v}|
>  \frac{\eta}{9}
\right) \\
&\leq& \left(\begin{array}{c}
p \\
2s \\
\end{array}
\right) 6^{2s} \left\{ n \exp\left(
-\frac{(\eta n)^\gamma}{ (9 \tilde{K})^\gamma s^{\gamma}C_1}
\right)\right.\\
&&\left.+ \exp\left(
-\frac{\eta^2 n^2}{81^2 \tilde{K}^2s^{2} C_2\{1+\tilde{K} n s^{5/2}\}}
\right)\right.
\\
&& \left.+ \exp\left(
-\frac{\eta^2 n}{81^2 \tilde{K}^2 s^2 C_3}
\right) \right\}\\
&\leq& \exp\left(2s \log 6p + \log n - \frac{(\eta n)^\gamma}{(9 \tilde{K})^\gamma s^{\gamma}C_1}\right)\\
&& + \exp\left(2s \log 6p -\frac{\eta^2 n^2}{81^2 \tilde{K}^2s^{2} C_2\{1+\tilde{K} n s^{5/2}\}}\right)\\
&& + \exp\left(
2s \log 6p - \frac{\eta^2 n}{81^2 \tilde{K}^2 s^2 C_3}
\right)\\
&\to& 0,\quad n \to \infty,
\end{eqnarray*}
where $\mathcal{B}_0(2s) = \{\bm{v} \in \mathbb{R}^p : 
\sum_{j=1}^p 1_{\{v_j \not=0\}} \leq 2s\}$.
This completes the proof.
\end{itemize}
\end{proof}

\subsection{Proofs for Section 4}\label{subsec: proof of sec4}
\begin{proof}[Proof of Proposition \ref{prop gradient diffusion}]
Since it holds that
\begin{eqnarray*}
\Prob(\|\psi_n^{(1)}(\bm{\theta}_0)\|_\infty > {\lambda}_n)
&\leq& \frac{1}{{\lambda}_n} \left\{
\E\left[
\|\bm{A}_n\|_\infty
\right]
+ \E\left[
\left\|
\bm{B}_n
\right\|_\infty
\right]
\right\},
\end{eqnarray*}
where 
\begin{eqnarray*}
\bm{A}_n 
&=& \frac{1}{n \Delta_n} \sum_{k=1}^n
\frac{\partial}{\partial \bm{\theta}}S(\bm{Y}_{t_{k-1}^n};\bm{\theta}_0)
\int_{t_{k-1}^n}^{t_k^n}
\{
S(\bm{Y}_s;\bm{\theta}_0)
-S(\bm{Y}_{t_{k-1}^n};\bm{\theta}_0)
\} \dr s
\end{eqnarray*}
and
\begin{eqnarray*}
\bm{B}_n 
&=& \frac{1}{n \Delta_n} \sum_{k=1}^n
\frac{\partial}{\partial \bm{\theta}}S(\bm{Y}_{t_{k-1}^n};\bm{\theta}_0)
\int_{t_{k-1}^n}^{t_k^n}
\sigma(\bm{Y}_s; h_0) \dr W_s,
\end{eqnarray*}
it suffices to show that 
\[
\E\left[
\left\|
\bm{A}_n
\right\|_\infty
\right]=o({\lambda}_n),\quad n \to \infty
\]
and that
\[
\E\left[
\left\|
\bm{B}_n
\right\|_\infty
\right]=o({\lambda}_n),\quad n \to \infty.
\]
Under our assumptions, we have 
\begin{eqnarray*}
\E\left[
\left\|
\bm{A}_n
\right\|_\infty
\right]
&\leq& \frac{1}{n \Delta_n} \sum_{k=1}^n 
\E\left[
\left\|
\frac{\partial}{\partial \bm{\theta}}S(\bm{Y}_{t_{k-1}^n};\bm{\theta}_0)
\right\|_\infty \left|
\int_{t_{k-1}^n}^{t_k^n}
\{
S(\bm{Y}_s;\bm{\theta}_0)
-S(\bm{Y}_{t_{k-1}^n};\bm{\theta}_0)
\} \dr s
\right|
\right]\\
&\leq& \frac{1}{n \Delta_n} \sum_{k=1}^n 
\E\left[
\left\|
\frac{\partial}{\partial \bm{\theta}}S(\bm{Y}_{t_{k-1}^n};\bm{\theta}_0)
\right\|_\infty
 \int_{t_{k-1}^n}^{t_k^n}
\left|
S(\bm{Y}_s;\bm{\theta}_0)
-S(\bm{Y}_{t_{k-1}^n};\bm{\theta}_0)
\right| \dr s
\right] \\
&\leq& \frac{L_S s}{n \Delta_n} \sum_{k=1}^n 
\E\left[
\left\|
\frac{\partial}{\partial \bm{\theta}}S(\bm{Y}_{t_{k-1}^n};\bm{\theta}_0)
\right\|_\infty\right.\\
&& \left. \cdot \int_{t_{k-1}^n}^{t_k^n}
\left\{
|X_s - X_{t_{k-1}^n}| + \|\bm{Z}_s - \bm{Z}_{t_{k-1}^n}\|_\infty
\right\} \dr s
\right] \\
&\leq&\frac{L_S s}{n \Delta_n} \sum_{k=1}^n 
\E\left[
\left\|
\frac{\partial}{\partial \bm{\theta}}S(\bm{\alpha}_0^\top \phi(X_{t_{k-1}^n}), \bm{\beta}_0^\top \bm{Z}_{t_{k-1}^n})
\right\|_\infty^2\right]^{1/2}\\
&&\cdot \E\left[\left| \int_{t_{k-1}^n}^{t_k^n}
\left\{
|X_s - X_{t_{k-1}^n}| + \|\bm{Z}_s - \bm{Z}_{t_{k-1}^n}\|_\infty
\right\} \dr s\right|^2
\right]^{1/2} \\
&\leq& 2^{1/2}L_S D_2^{1/2} s K_{p} \Delta_n^{1/2},
\end{eqnarray*}
which implies
\[
\frac{1}{{\lambda}_n}\E[\|A_n\|_\infty]
\leq \frac{s K_{p} \Delta_n^{1/2}}{{\lambda}_n} \cdot2L_S D_2^{1/2}
\to 0,\quad n \to \infty.
\]
It follows from Lemma \ref{maximal ineq}
that 
\begin{eqnarray*}
\E\left[
\left\|
\bm{B}_n
\right\|_\infty
\right]
&=& \E\left[
\max_{1 \leq i \leq p} \left|
\sum_{k=1}^n \xi_k^i
\right|
\right] \\
&\leq& 2\sqrt{2} \sqrt{\E\left[
\sum_{k=1}^n \max_{1 \leq i \leq p} \left|
\xi_k^i
\right|^2
\right]}
\sqrt{\log(1+p)},
\end{eqnarray*}
where 
\[
\xi_k^i =\frac{1}{n \Delta_n} \frac{\partial}{\partial \bm{\theta}_i}S(\bm{Y}_{t_{k-1}^n};\bm{\theta}_0)\int_{t_{k-1}^n}^{t_k^n} \sigma(\bm{Y}_s; h_0) \dr W_s,\quad i=1,\ldots,p.
\]
We have 
\begin{eqnarray*}
\lefteqn{\E\left[
\sum_{k=1}^n \max_{1 \leq i \leq p} \left|
\xi_k^i
\right|^2
\right]}\\
&\leq& \frac{1}{n^2 \Delta_n^2}\sum_{k=1}^n\E\left[
\left\|\frac{\partial}{\partial \bm{\theta}} S(\bm{Y}_{t_{k-1}^n};\bm{\theta}_0)\right\|_\infty^2
\E\left[
\left\{\int_{t_{k-1}^n}^{t_k^n}
\sigma(\bm{Y}_s; h_0)
\dr W_s\right\}^2
\middle|\mathcal{F}_{t_{k-1}^n}\right]
\right]\\
&=& \frac{1}{n^2 \Delta_n^2}\sum_{k=1}^n\E\left[
\left\|\frac{\partial}{\partial \bm{\theta}} S(\bm{Y}_{t_{k-1}^n};\bm{\theta}_0)\right\|_\infty^2
\E\left[
\left\{\int_{t_{k-1}^n}^{t_k^n}
\sigma^2(\bm{Y}_s; h_0)
\dr s\right\}
\middle|\mathcal{F}_{t_{k-1}^n}\right]
\right]\\
&\leq& K_{p}^2\frac{1}{n^2 \Delta_n^2}\sum_{k=1}^n
\E\left[
\left\{\int_{t_{k-1}^n}^{t_k^n}
\sigma^2(\bm{Y}_s; h_0)
\dr s\right\}^2
\right]^{1/2}\\
&\leq& \frac{1}{n \Delta_n} K_{p}^4.
\end{eqnarray*}
Therefore, we obtain  
\[
\E\left[
\|\bm{B}_n\|_\infty
\right]
\leq 2\sqrt{2}K_p^2 \sqrt{\frac{\log (1+p)}{n \Delta_n}}.
\]
This completes the proof.
\end{proof}
\begin{proof}[Proof of Proposition \ref{linear diffusion maximal ineq}]
The assertion (i) follows from Assumption \ref{regularity diffusion}-(ii).
To prove the assertion (ii),
it suffices to prove that 
there exists a constant $0 < \eta < f_\infty/2$ such that 
\begin{equation}\label{Wn inequality diffusion}
\Prob\left(
\sup_{\bm{v} \in C_{T_0}\cap \mathbb{S}_p}
|\bm{v}^\top \bm{W}_n \bm{v}| > \eta
\right) \to 0,\quad
n \to \infty,
\end{equation}
where $\bm{W}_n = \bm{V}_n^{(1)}-\bm{V}$.
Let $\bm{Y}(n)$ be the matrix defined by 
\[
\bm{Y}(n) = (\bm{Y}_0,\ldots,\bm{Y}_{t_{n-1}^n}).
\]
Then, it holds that 
\[
\bm{V}_n^{(1)} = \frac{1}{n} \bm{Y}(n) \bm{Y}(n)^\top,\quad
\bm{V} = \frac{1}{n}\E[\bm{Y}(n) \bm{Y}(n)^\top].
\]
For every $\bm{v} \in C_{T_0}\cap \mathbb{S}_p$, we have 
\[
\bm{v}^\top \bm{W}_n \bm{v}
= \frac{1}{n}\{\|\bm{Y}(n)^\top \bm{v}\|_2^2 - \E[\|\bm{Y}(n)^\top \bm{v}\|_2^2]\}.
\]
Noting that $\bm{Y}(n)^\top \bm{v}$ is a centered Gaussian random variable with covariance matrix
\[
\bm{Q}_n = \left(\begin{array}{ccc}
\bm{v}^\top \E[\bm{Y}_0 \bm{Y}_0^\top]\bm{v}& \cdots& \bm{v}^\top \E[\bm{Y}_0 \bm{Y}_{t_{n-1}^n}^\top]\bm{v} \\
\vdots & \ddots & \vdots \\
\bm{v}^\top \E[\bm{Y}_{t_{n-1}^n}\bm{Y}_0^\top]\bm{v}& \cdots&
\bm{v}^\top  \E[\bm{Y}_{t_{n-1}^n}\bm{Y}_{t_{n-1}^n}^\top]\bm{v}
\end{array}
\right),
\]
we can use the Hanson--Wright inequality
(see \cite{rudelson2013hanson} for details) to deduce that 
for every $\xi>0$,
\[
\Prob\left(
\frac{1}{n} \left|
\|\bm{Y}(n)^\top \bm{v}\|_2^2 - \E[\|\bm{Y}(n)^\top \bm{v}\|_2^2]
\right|
> \xi \|\bm{Q}_n\|_2
\right)
\leq 2 \exp(-cn (\xi \land \xi^2)).
\]
From a similar calculation to the 
proof of Lemma 10 in Supplementary material of 
\cite{wong2020lasso}, we have
\[
\|\bm{Q}_n\|_2 
\leq
\max_{0 \leq i \leq n-1} \sum_{l=0}^{n-1} \|\bm{\Sigma}_{Y}((l-i)\Delta_n)\|_2,
\]
where 
\[
\bm{\Sigma}_{Y}((l-i)\Delta_n)
 =C(t_l^n , t_i^n) 
\]
From Assumption 4.7, it holds that 
\[
\|\bm{\Sigma}_{Y}((l-i)\Delta_n)\|_2
= \|\exp(|l-i|\Delta_n \bm{A})\bm{V}\|_2.
\]
Since $\bm{A}$ is diagonalizable, there exists a 
non-singular matrix $\bm{P}$ with $\|\bm{P}\|_2=1$ such that 
\[
\bm{P}^\top \bm{A} \bm{P} = \bm{D},
\]
where $\bm{D}$ is a diagonal matrix consists of eigenvalues of $\bm{A}$.
Therefore, we have 
\begin{eqnarray*}
\|\bm{\Sigma}_{Y}((l-i)\Delta_n)\|_2
&=& \|\exp(|l-i|\Delta_n \bm{A})\bm{V}\|_2\\
&\leq& \|\exp(|l-i|\Delta_n \bm{D})\|_2 \|\bm{V}\|_2\\
&\leq& \exp(|l-i|\Delta_n \varsigma_1) K_p,
\end{eqnarray*}
which implies 
\begin{eqnarray*}
\|\bm{Q}_n\|_2 
&\leq&
\sum_{l=0}^\infty \exp(l \Delta_n \varsigma_1)K_p.
\end{eqnarray*}
Then, for 
\[\xi = \frac{\eta}{\sum_{l=0}^\infty \exp(l \Delta_n \varsigma_1)K_p},
\]
we have 
\[
\Prob(|\bm{v}^\top \bm{W}_n \bm{v}| > \eta)\leq
2 \exp(-cn (\xi \land \xi^2)).
\]
Therefore, as well as Proof of Proposition 3.1 in \cite{basu2015regularized} and the $\epsilon$-net argument provided in \cite{vershynin2018high}, 
we have
\begin{eqnarray*}
\Prob \left(
\sup_{\bm{v} \in C_{T_0}\cap \mathbb{S}_p} |\bm{v}^\top \bm{W}_n \bm{v}|
> \eta
\right)
&\leq&
\Prob \left(
9\sup_{\bm{v} \in \mathcal{B}_0(2s)} |\bm{v}^\top \bm{W}_n \bm{v}|
> \eta
\right) \\
&\leq&
\Prob \left(
\sup_{\bm{v} \in \mathcal{B}_0(2s)} |\bm{v}^\top \bm{W}_n \bm{v}|
>  \frac{\eta}{9}
\right) \\
&\leq& \left(\begin{array}{c}
p \\
2s \\
\end{array}
\right) 6^{2s} \cdot 2 \exp(-cn (\xi/9 \land \xi^2/9^2))\\
&\leq& 2 \exp(2s \log 6p - cn (\xi/9 \land \xi^2/9^2))\\
&\to& 0,\quad n \to \infty,
\end{eqnarray*}
which ends the proof.
\end{proof}

\section*{Acknowledgements}

This work was
supported by JSPS KAKENHI Grant Numbers 21K13271(K.F.), 21K13836 (K.T.).
\bibliographystyle{chicago}
\bibliography{DS-ref2}
\end{document}